\documentclass[a4paper, 11pt]{article}
\usepackage[english]{babel}
\usepackage[utf8]{inputenc}

\usepackage{amsmath,amsthm,amsfonts}
\usepackage{mathtools}
\usepackage{graphicx}
\usepackage{subcaption}
\usepackage{algorithm}
\usepackage{algpseudocode}
\usepackage{color}
\usepackage{multirow}

\newtheorem{theorem}{Theorem}
\newtheorem{lemma}{Lemma}

\graphicspath{ {./figs/}}

\title{Multiscale Graph Reduction for Heterogeneous \\ and Anisotropic Discrete Diffusion Processes}

\author{
Maria Vasilyeva 
\thanks{Department of Mathematics and Statistics, Texas A\&M University-Corpus Christi,   Corpus Christi, TX, USA\\
Email: {\tt maria.vasilyeva@tamucc.edu}.} 
\and
James Brannick
\thanks{Department of Mathematics, The Penn State University, University Park, PA, USA.\\
Email: {\tt jjb23@psu.edu}.} 
\and
Ben S. Southworth
\thanks{Theoretical Division, Los Alamos National Laboratory, Los Alamos, NM, USA.\\
Email: {\tt southworth@lanl.gov}.} 
}

\begin{document}

\maketitle

\begin{abstract}
We present multiscale graph-based reduction algorithms for upscaling heterogeneous and anisotropic diffusion problems. The proposed coarsening approaches begin by constructing a partition of the computational domain into a set of balanced local subdomains, resulting in a standard type of domain decomposition. Given this initial decomposition, general coarsening techniques based on spectral clustering are applied within each subgraph to accurately identify the key microscopic features of a given system. The spectral clustering algorithm is based on local generalized eigen-decompositions applied to the signed graph Laplacian. The resulting coarse-fine splittings are combined with two variants of energy-minimizing strategies for constructing coarse bases for diffusion problems. The first is an unconstrained minimization formulation in which local harmonic extensions are applied column-wise to construct multi-vector preserving interpolation in each region, whereas the second approach is a variant of the constrained energy minimization formulations derived in the context of non-local multi-continua upscaling techniques. We apply the resulting upscaling algorithms to a variety of tests arising from the graph Laplacian, including diffusion in the perforated domain, channelized media, highly anisotropic settings, and discrete pore network models to demonstrate the potential and robustness of the proposed coarsening approaches. We show numerically and theoretically that the proposed approaches lead to accurate coarse-scale models and approximations {relative to the corresponding fine-scale discrete system and corresponding approximation.  Initial results of a two-level solver applied to these same test problems that uses in addition a block smoother with coloring are also provided.  In this latter solver setting, the two approaches can be seen as extensions of the two main variants of standard algebraic multigrid methods, classical AMG and smoothed aggregation, respectively.} 
\end{abstract}

\section{Introduction}

 Diffusion processes in heterogeneous and anisotropic media arise in many applications, including subsurface flow, composite materials, and image and data analysis. Solving the resulting discrete systems becomes particularly challenging when the diffusion coefficient  varies strongly or exhibits anisotropy, often spanning several orders of magnitude across the domain. This variability increases the complexity of the resulting linear systems, making them harder to solve efficiently. To reduce the system size and computational cost, various upscaling and multiscale methods have been developed, which yield smaller, coarse-scale systems. Among these, domain decomposition, multiscale, and multigrid methods are widely used, often demonstrating high efficiency. Numerous strategies for constructing a local basis (interpolation) have been developed for heterogeneous problems, mostly in the context of standard (uniform) coarsening.  To compensate for simple coarsening, one generally uses energy minimization techniques,  spectral methods, or both since these approaches generate local harmonic bases within local subdomains that can, to some degree, account for the use of coarse variables that do not align with the heterogeneity or anisotropy of the given problem.Each such coarse basis vector in these approaches is associated with a given subdomain, and its nonzero support is in that entire subdomain. Hence, even when effective, spectral methods lead to fill-in on coarse levels and large operator complexities overall, especially when combined with aggressive coarsening, e.g. \cite{southworth2026algebraic,al2019class}.  {We note the recent local and global spectral AMG type methods developed in the papers ~\cite{vasilyeva2025generalized, vasilyeva2025msaniso,  southworth2026algebraic}. 

{Standard algebraic multigrid (AMG) methods, such as classical AMG and smoothed aggregation, attempt to coarsen the degrees of freedom more gradually and in a way that directly accounts for these local features (e.g., anisotropy) in the choice of the coarse variables.  However, their performance can degrade, especially when highly variable anisotropy is present, since they make use of pointwise smoothers and typically are not able to fully capture the direction of the anisotropic medium in such cases since their design focuses on solving $M$-matrix problems as efficiently as possible.  Energy minimization AMG approaches with pointwise smoothers have been effectively designed for non grid-aligned anisotropic problems with mild to moderate anisotropy ratios~\cite{JacobAniso}.  At the same time, the paper~\cite{krzysik2026overlapping}, 
shows that using additive Schwarz smoothers is not robust when combined with different simple geometric interpolation for general anisotropic problems. We finally mention again the papers on spectral AMG, in the Bootstrap AMG~\cite{BAMG2010} and domain decomposition settings~\cite{southworth2026algebraic},  
that solve global and local spectral problems, respectively, to define coarse basis.  Notably, both require a linear growth in cycle complexity with respect to the number of levels in order to produce effective solvers, even when Schwarz smoothers are used in the solver.}
From an application standpoint, graph-based physical modeling has traditionally relied on homogenization and upscaling techniques \cite{iliev2010fast, reichold2009vascular}. Recent advances improve local accuracy through multiscale methods, including mortar coupling of network and continuum models \cite{balhoff2008mortar} and the heterogeneous multiscale method \cite{chu2012multiscale}. In \cite{barker2017spectral}, coarsening for a mixed form of the Laplacian has shown promise in applications to continuum-scale reservoir simulations. Localized orthogonal decomposition has also been explored for fiber-based materials \cite{kettil2020numerical, gortz2022iterative}. 
However, most multiscale methods still assign a single macroscopic variable per local subdomain and rely on structured coarse grids, which limits their effectiveness in complex networks.

This paper develops Multiscale Graph Reduction (MsGR) techniques for coarsening diffusion processes on heterogeneous and anisotropic structures. To capture the complex discrete structure of the materials, the method introduces local communities with distinct local behaviors that are identified using spectral clustering techniques applied within subdomains.  The initial subdomains are constructed as in domain decomposition and spectral AMG methods using, for example, METIS ~\cite{karypis1997metis}. In the proposed scheme, the clustering is derived in each subdomain using a local generalized eigen-decomposition involving the signed graph Laplacian and its diagonal. We emphasize that using the signed graph Laplacian is a purely algebraic way to construct local {definite} matrices on subdomains or aggregates. 
Unlike traditional local spectral models that make use of multiple abstract basis functions on each subdomain, our approach assigns explicit macroscopic variables to the dominant modes identified via the local spectral problems 
and restricts their supports to the resulting clusters; hence, the coarse basis functions are potentially significantly more sparse than the ones used in spectral methods, with the supports of a basis vector defined over its entire associated subdomain.  
We consider two approaches for constructing multiscale basis functions and the corresponding prolongation operators. 
{These two choices represent two complementary interpretations of coarse variables in AMG.} 
The first is a classical AMG {coarse-fine (CF)}-type approach, which selects centroids of the clusters within subdomains as coarse nodes and builds the interpolation operator using node splitting and a label propagation scheme (i.e., localized approximations of ideal interpolation). 
{The second approach is the {multicontinuum (MC)-approach}, where local communities (clusters) are treated as coarse variables, and basis functions are obtained by solving local constrained energy minimization problems that capture the average behavior within each local community, resempling smoothed aggregation AMG. Although the MC approach is more accurate in many of the numerical tests, the CF approach remains important because it is simpler, directly connected to standard C/F splittings and ideal interpolation used in classical AMG, and provides a natural baseline for more advanced two-grid extensions, e.g., the MC approach or adaptive and bootstrap AMG sovlers. As such, a main contribution of this paper is the development of a coarsening algorithm that generates a coarse representation in which the coarse variables follow the physics and thus retain a clear interpretation and visualization on coarse levels.  In this way, the proposed approach can be seen as an extension of AMG methods, which use heuristics based on $M$-matrix assumptions to coarsen, leading to their poor performance for non-grid aligned anisotropic diffusion systems as we consider here. }

{ To further illustrate the potential of both proposed coarsening approaches, we provide results of a two-level solver with pointwise Jacobi and colored block Schwarz smoothers for the same test problems. The proposed overlapping Schwarz iterations incorporate local subspace corrections with coloring for parallelization and to avoid using a partition of unity and is motivated by the local residual correction iterations explored in~\cite{efendiev2013generalized, tyrylgin2021multiscale} and more recently in~\cite{vasilyeva2025implicit}.
The present work develops an algebraic algorithm that uses graph partitioning and coloring to construct colored Schwarz smoothers for arbitrary coarse partitions. }

The organization of this paper is as follows.  In Section 2, we introduce the weighted graph Laplacian and present different discrete formulations for the anisotropic and diffusive processes. In Section 3, we present the construction of the coarse-scale model, which is based on domain decomposition and local spectral clustering. The numerical results are presented in Section 4 for multiple challenging problems, including (i) graph representations arising from finite element discretizations of heterogeneous perforated media; (ii) highly anisotropic heat flow, and (iii) high-contrast pore network model. The paper ends with a conclusion. 

\section{Problem formulation}
The aim of this paper is to design multiscale graph reduction  (MsGR) algorithms for coarsening and upscaling. The methods are based on constructing local weighted graph Laplacians, and using these graph-based representations of the problem to implement aggregation, partitioning, and spectral methods that can be used to coarsen a variety of discrete models.  We consider coarsening two forms of heterogneous and anisotropic diffusion problems: (i) a classical finite element formulation, derived from a continuous diffusion model (elliptic equation), and (ii) a discrete graph model based on the pore network framework~\cite{blunt2013pore}.

\subsection{An elliptic model problem.}
The traditional formulation of the diffusion model is given by the elliptic problem
\begin{equation} 
\label{eq:ell}
- \nabla \cdot (K(x) \nabla u) = q \quad \text{in } \Omega,
\end{equation}
with Dirichlet boundary conditions
\begin{equation}
u = g \quad \text{on } \partial \Omega,
\end{equation}
where $q$ is the given source term, {$g$ is the prescribed Dirichlet boundary value,} and $K$ is a symmetric, positive-definite tensor field that may vary spatially and exhibit strong anisotropy or heterogeneity, and the computational domain is denoted by $\Omega$. 
{Figure \ref{fig:perf} shows the heterogeneous perforated domain $\Omega$ used later in Test 1. The blue region is the background subdomain $\Omega_1$ with coefficient $K=K_1$ and the red channels form the high-conductivity subdomain $\Omega_2$ with coefficient $K=K_2$. The figure illustrates both geometric heterogeneity, through the perforations, and coefficient heterogeneity, through the two material regions.}

\begin{figure}[h!]
\centering
\includegraphics[width=0.4\linewidth]{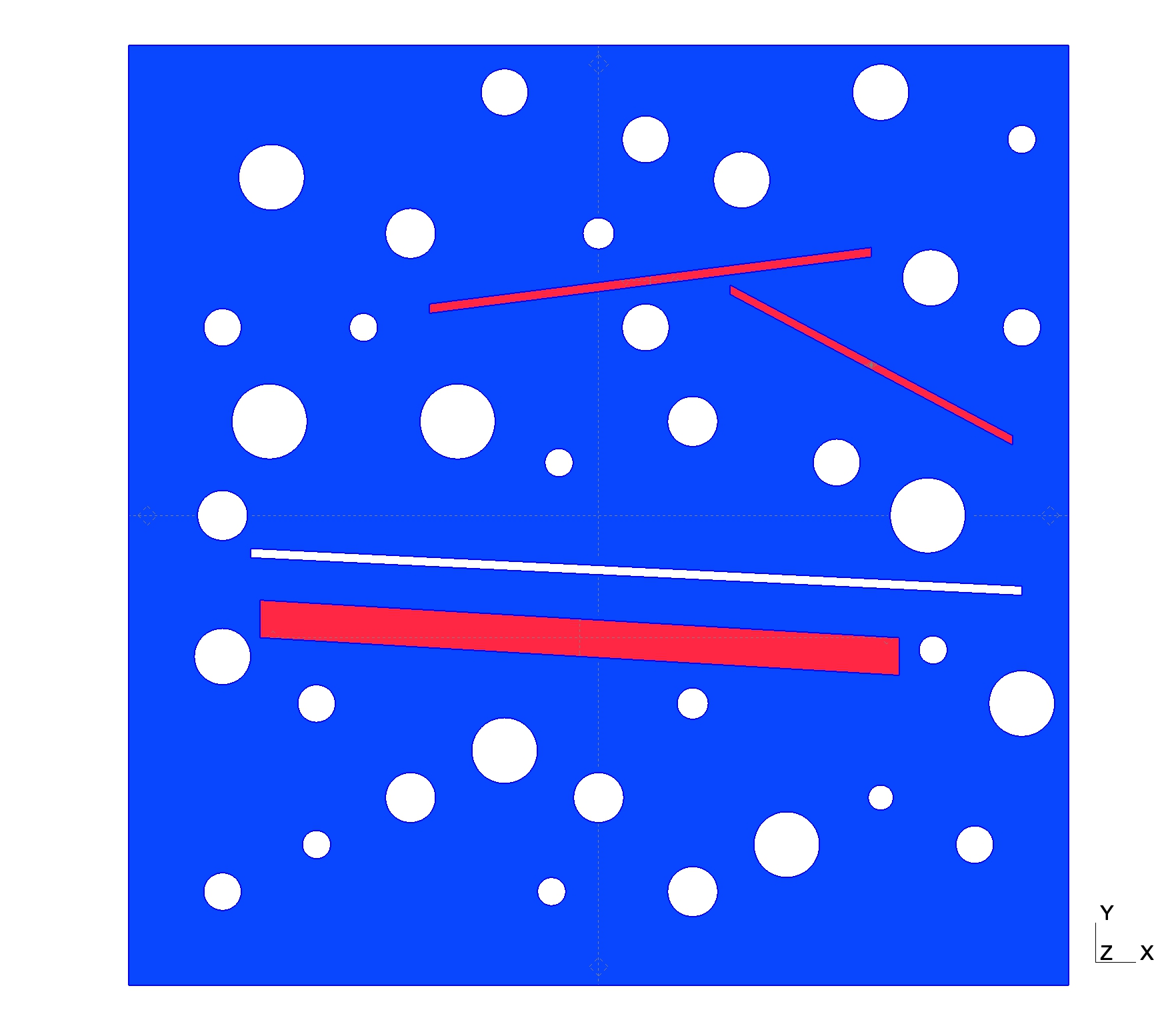}
\includegraphics[width=0.4\linewidth]{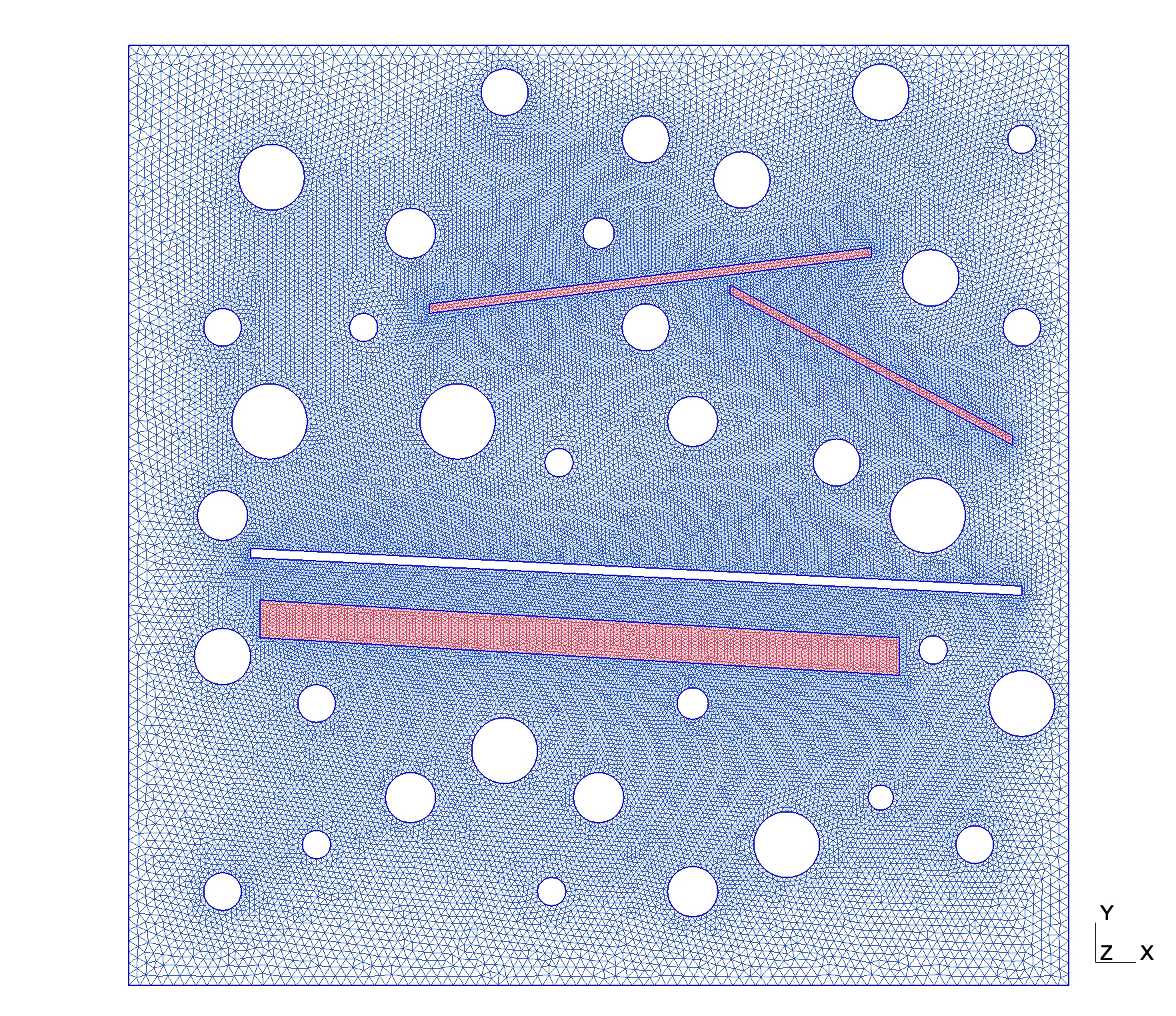}
\caption{Example heterogeneous perforated domain $\Omega$ and mesh $\mathcal{T}_h$. 
{Left: coefficient/material layout for Test 1, where the blue background subdomain has coefficient $K=K_1$, the red channel subdomain has coefficient $K=K_2$ (with $K_2 \gg K_1$ in the high-contrast tests). Right: finite element mesh conforming to the perforations and material interfaces.}}
\label{fig:perf}
\end{figure}

To approximate this problem numerically, we use a standard conforming finite-dimensional subspace $V_h \subset H_0^1(\Omega)$, consisting of piecewise linear functions defined on a mesh $\mathcal{T}_h$. 
The finite element approximation seeks $u_h \in V_h$ such that
\[
\int_{\Omega} K \nabla u_h \cdot \nabla v_h \, dx = \int_{\Omega} q v_h \, dx, \quad \forall v_h \in V_h.
\]

Let \( \{ \phi_i \}_{i=1}^n \) denote the linear basis functions of $V_h${, where $n=\dim(V_h)$ is the number of fine-scale degrees of freedom (DOFs).}  Then, the approximate solution can be expressed as
\[
u_h(x) = \sum_{j=1}^n u_j \phi_j(x),
\]
and the variational problem reduces to a linear system
\[
A u = f,
\]
where $A \in \mathbb{R}^{n \times n}$ is the stiffness matrix, $u \in \mathbb{R}^n$ is the vector of unknown nodal values, and $f \in \mathbb{R}^n$ is the load vector. The entries of $A$ and $f$ are given by
\[
A = \{a_{ij}\}, \quad a_{ij} = \int_{\Omega} K \ \nabla \phi_j \cdot \nabla \phi_i \ dx, \quad
f = \{f_{i}\}, \quad  f_i = \int_{\Omega} q \ \phi_i \ dx.
\]

For the special case of free flux boundary conditions, first order elements, and isotropic $K(x)=1$, we can interpret the finite element stiffness matrix exactly as a weighted graph Laplacian. More generally, we associate a graph and corresponding weight matrix to the off-diagonal entries in the stiffness matrix, with the graph structure induced by the connectivity of the finite element basis functions and vertices for $\mathbb{P}_1$ elements. Then $G = (\mathcal{V}, \mathcal{E})$ is a signed undirected graph, on which we define a weighted graph Laplacian $L\in\mathbb{R}^{n\times n}$. 
{Let $\mathcal{N}(i)=\{j\in\mathcal{V}:(i,j)\in\mathcal{E}\}$ denote the set of graph neighbors of node $i$. Then }
\begin{equation}\label{eq:laplacian}
L_{ij} = 
\begin{cases}
\sum_{k \in \mathcal{N}(i)} |a_{ik}|, & \text{if } i = j, \\
-a_{ij}, & \text{if } (i,j) \in \mathcal{E}.
\end{cases}
\end{equation}
The edge weights reflect the anisotropic and heterogeneous properties of the tensor $K(x)$.  
A key part of utilizing local spectral problems is imposing some sort of natural boundary condition, while maintaining definiteness of the local matrices. If one utilizes element information, this can be assembled directly as in, e.g., the Generalized Multiscale Finite Element Method (GMsFEM) \cite{efendiev2013generalized, vasilyeva2025generalized} or element-based AMG \cite{brezina2001algebraic}. Algebraic approaches are more complex and can be posed a finding a symmetric positive semi-definite (SPSD) splitting \cite{al2019class}. Here, we recognize the signed graph Laplacian as a simple method to maintain proper off-diagonal structure and ensure definiteness. Although local local signed graph Laplacians sacrifice the splitting property of local matrices to the global matrix as in \cite{al2019class}, in the context of identifying coarse variables we find it to be a fast and robust fully algebraic approach. 

\subsection{Discrete pore network model}

The pore network model (PNM) that we consider is a computational framework for simulating flow and transport processes in porous media by representing the complex geometry of a porous matrix as a graph/network structure \cite{blunt2013pore, valvatne2004predictive}. 
In PNM, the pore space is represented as a network of interconnected nodes, where mass conservation equations are defined over the network to simulate multiphase flow and transport processes. The geometry and connectivity of the network are typically extracted from high-resolution imaging (e.g., micro-CT), and flow between pores relates local conductance to pore geometry and fluid properties.  
In Figure  \ref{fig:porenet}, we generate multiple images of different scales, combine them, and apply a network generation algorithm from the {OpenPNM} library to extract a multiscale pore network \cite{gostick2016openpnm}. 
{The left panel shows the two-scale pore image used as input. The right panel shows the corresponding extracted network: graph vertices represent pore bodies and graph edges represent throats, with locally dense connected regions producing high-contrast conductance patterns in the resulting weighted graph.}
{We note that we do not attempt to quantify the modeling error introduced by image resolution, image segmentation, or the pore-network extraction procedure. In this work, the extracted pore network is treated as the given fine-scale discrete model, and the reported coarse-scale errors measure the reduction error relative to this fine-scale graph model.}

\begin{figure}[h!]
\centering
\includegraphics[width=0.48\textwidth]{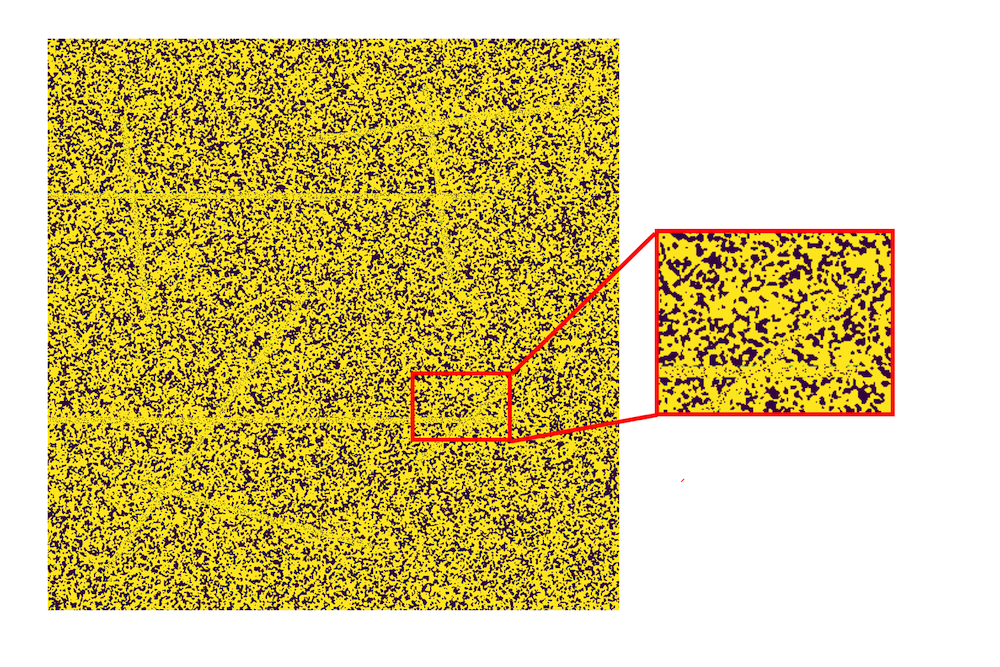}
\includegraphics[width=0.48\textwidth]{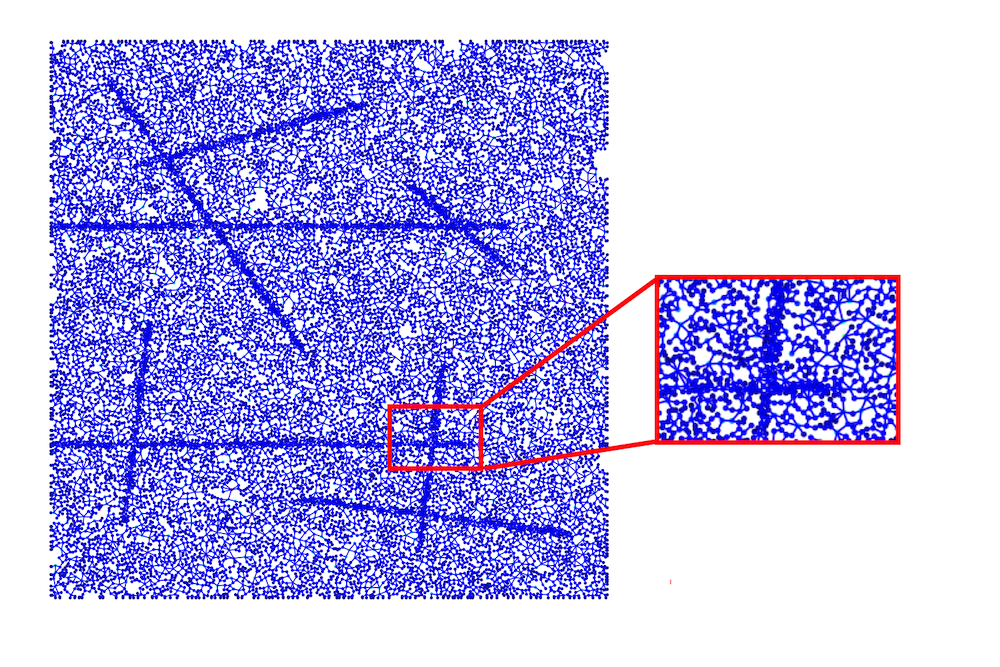}
\caption{Pore-scale structure with two scales and corresponding high-contrast network. 
{Left: input pore-scale image formed by combining structures at two spatial scales. Right: pore network extracted from the image, where vertices and edges represent pore bodies and throats. The variation in pore and throat sizes leads to heterogeneous hydraulic conductances $w_{ij}$ in the weighted graph model.}}
\label{fig:porenet}
\end{figure}

The pore network model can be represented as a weighted undirected graph $G = (\mathcal{V}, \mathcal{E})$, where each node $i \in \mathcal{V}$ corresponds to a pore body at position $x_i$ and each edge $(i,j) \in \mathcal{E}$ represents a throat connecting the two pores with weight $w_{ij} \geq 0$.
Therefore, the system of flow equations (mass balance) can be written using the weighted graph Laplacian matrix $L \in \mathbb{R}^{n \times n}$ defined as in \eqref{eq:laplacian}, where $n = |\mathcal{V}|$ is the number of pores.  

Here, each throat (edge) is assigned a hydraulic conductance $w_{ij}\geq 0$, which characterizes its ability to transmit fluid between adjacent pores $i$ and $j$. 
For a Newtonian fluid, assuming laminar flow and a circular cross-section, the conductance is typically given by the Hagen–Poiseuille law \cite{blunt2013pore, gostick2017versatile, gostick2016openpnm, kettil2020numerical}
\[
w_{ij} = \frac{\pi r_{ij}^4}{8 \mu l_{ij}},
\]
where $r_{ij}$ is the radius of the throat, $l_{ij}$ is the length of the throat, and $\mu$ is the dynamic viscosity of the fluid.  We note that in this setting the weight matrix $W$ is assumed to be non-negative and, thus, the resulting approximation results in the standard (unsigned) weighted graph Laplacian system.

In the pore network model, the analogue of a Robin boundary condition is imposed in a pointwise manner at boundary nodes. 
If $\mathcal{V}_{\mathrm{bdy}} \subset \mathcal{V}$ denotes the set of boundary vertices and $\alpha_i \geq 0$ the Robin parameter, then at a boundary node $i \in \mathcal{V}_{\mathrm{bdy}}$ we have 
\[
(Au)_i =  (Lu)_i + \alpha_i u_i = b_i + \alpha_i  g = f_i,
\]
and for the interior nodes $i \in \mathcal{V}/\mathcal{V}_{\mathrm{bdy}}$
\[
(Lu)_i = b_i,
\]
where 
$u_i$ is the pressure at pore $i$ and 
$b_i$ encodes the source/sink terms. 
Equivalently, the weighted graph Laplacian is modified by adding a diagonal contribution on boundary nodes:
\begin{equation}
    A = L + B, \quad 
B_{ii} = 
\begin{cases}
\alpha_i, & i \in V_{\mathrm{bdy}}, \\
0, & i \in V \setminus V_{\mathrm{bdy}}.
\end{cases}
\label{eq:sys-f}
\end{equation}
{Here, $A$ is the graph-based system matrix after incorporating the boundary contribution $B$, and $f$ is the right-hand side that contains both the source/sink term $b$ and the boundary data $g$.}

\section{MsGR: Multiscale Graph Reduction}

 Upscaling and multiscale methods use an aggressive coarsening strategy to construct a relatively small coarse model of the underlying processes \cite{allaire1992homogenization, hou1997multiscale}. For example, a uniform coarse grid is often utilized in upscaling/multiscale methods, which is also used in the geometric multigrid setting. Furthermore, upscaling and multiscale methods are based on solving local problems: in representative volumes for periodic media, and in each local subdomain for the non-periodic case. In practice, one is interested in equally sized local domains for load-balanced calculations, where a balanced grid partition is usually connected with parallel computations and domain decomposition methods.

In this paper, we propose a coarsening strategy that includes local spectral clustering to define local communities within each subdomain.  We use the clusters to enrich the multiscale space, as opposed to using the entire subdomain. As a result, the coarse degrees of freedom and resulting bases capture the complex microscale behavior, can be used with aggressive coarsening to obtain a physically meaningful coarse-scale graph representation, and typically lead to significantly lower complexities when compared to using multiple basis functions with supports on the complete subdomain.

The overall approach we propose is closely related to the multicontinuum theory that has been developed for high-contrast media in the context of the GMsFEMs \cite{efendiev2013generalized, vasilyeva2025generalized} and to the spectral AMG method originally introduced in~\cite{chartier2003spectral}, but with a refined coarsening strategy and approach for constructing the supports of coarse basis functions.  Generally, the MsGR approach can be seen as a hybrid approach that combines ideas from domain decomposition and AMG methods.  In terms of domain decomposition, MsGR refines the subdomains by forming clusters within each subdomain, thereby partitioning and adapting the subdomains in a way that captures local heterogeneity and anisotropy. And with respect to AMG coarsening, our approach both localizes the AMG coarsening process, restricting it to subdomains that preserve the mesh features locally, and refines AMG coarsening, which typically relies on heuristics based on the size of the matrix entries to define  maximal independent set algorithms used in coarsening. 

\subsection{Graph partitioning and local domains}

Similar to domain decomposition methods, we propose using graph partitioning algorithms like METIS \cite{karypis1997metis} or SCOTCH \cite{pellegrini1996scotch} to form the initial set of subdomains. These algorithms are designed to minimize the interface size between subdomains while maintaining a balance in subdomain sizes and the mesh topology locally within the subdomains.

\begin{figure}[h!]
\centering
\includegraphics[width=0.345\linewidth]{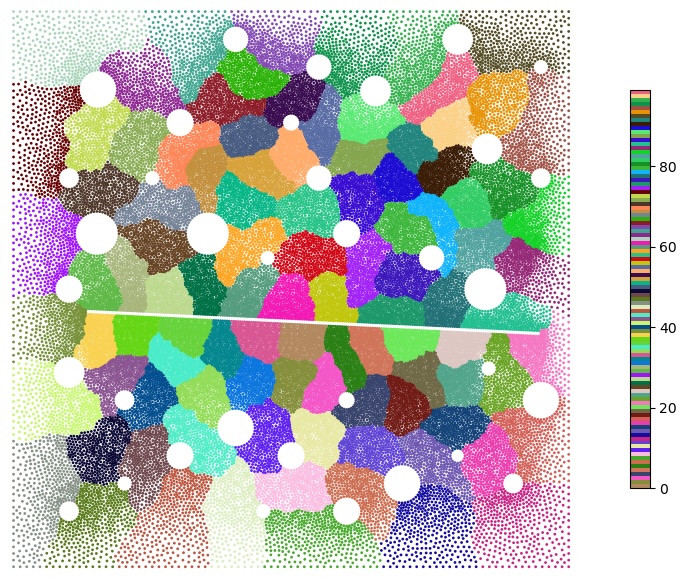}
\includegraphics[width=0.32\linewidth]{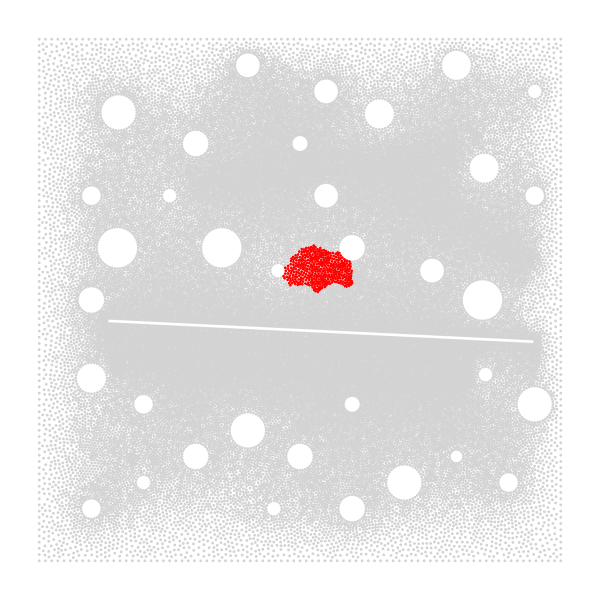}
\includegraphics[width=0.32\linewidth]{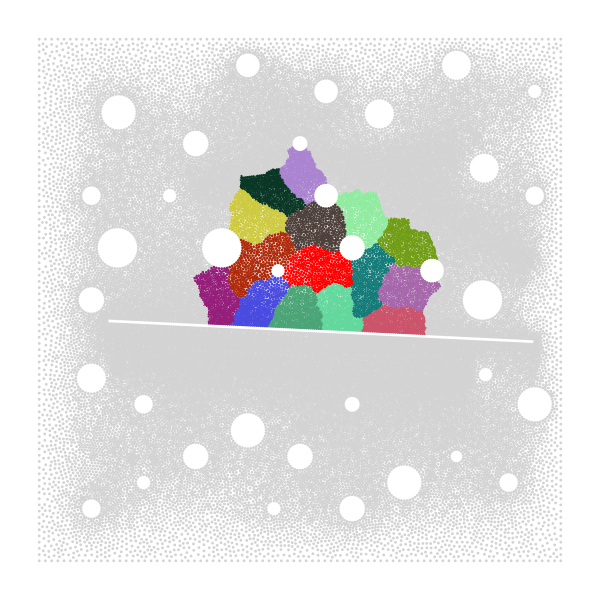}
\caption{Graph partitioning and local domains $\Omega_k$ with $N_{\Omega}=100$ (left). Illustration of oversampled subdomain $\Omega_k^+$. Target subdomain $\Omega_k$ (middle) and oversampled subdomains $\Omega_k^+$ based on distances $\delta_H$ (right). Red color is used for target subdomain $\Omega_k$}
\label{fig:dd}
\end{figure}

The proposed MsGR algorithm begins with a partition of the graph $G = (\mathcal{V}, \mathcal{E})$ into $N_{\Omega}$ disjoint subdomains
\[
\bigcup_{k=1}^{N_{\Omega}} \Omega_k = \mathcal{V}, \quad \Omega_k \cap \Omega_\ell = \emptyset \quad \text{for } k \neq \ell.
\]
Each subdomain $\Omega_k \subset \mathcal{V}$ defines a corresponding subgraph $G_k = (\Omega_k, \mathcal{E}_k)$, where $\mathcal{E}_k$ is the set of edges with both endpoints in $\Omega_k$:
\[
G_k = \left(\Omega_k, \mathcal{E}_k \right), \quad  \mathcal{E}_k = \left\{ (u, v) \in \mathcal{E} \mid u, v \in \Omega_k \right\}.
\]
In Figure \ref{fig:dd}, we plot a balanced graph partition of 100 subdomains. We used PyMetis (Python wrapper for the METIS library \cite{karypis1997metis}) for partitioning in this example. The DOFs correspond to the triangulation of the perforated domain constructed using the Gmsh library \cite{geuzaine2009gmsh}, and the mesh contains 80,746 cells, 122,012 edges, and 41,231 nodes. The partitioning to $N_{\Omega}=100$ subdomains generates subdomains of size $n^{\Omega_k} = |\Omega_k|  \approx 412$ (min=400, max=424).

To capture local interactions beyond $\Omega_k$, we define an oversampled region $\Omega_k^{+}$ by including all neighboring subdomains that contain nodes within a prescribed physical distance $\delta_H > 0$ from any node in $\Omega_k$. 
{Thus $\Omega_k^+$ is not formed by adding only the individual nodes that satisfy the distance criterion. Instead, it is defined as the union of the target subdomain and all neighboring subdomains that satisfy this criterion.}
Letting $\mathrm{dist}(u,v)$ denote the Euclidean distance between nodes $u$ and $v$,  {we first define the index set of selected subdomains by
\[
\mathcal{I}_k^+ \coloneqq \left\{ \ell \in \{1,\ldots,N_{\Omega}\} \;\middle|\; \exists\, u \in \Omega_k,\ \exists\, v \in \Omega_\ell \text{ such that } \mathrm{dist}(u, v) \leq \delta_H \right\}.
\]
The oversampled region is then defined as
\[
\Omega_k^{+} \coloneqq \bigcup_{\ell \in \mathcal{I}_k^+} \Omega_\ell.
\]
}
The corresponding oversampled subgraph is then defined as
\[
G_k^{+} \coloneqq \left(\Omega_k^{+}, \mathcal{E}_k^{+}\right), \quad \text{where } \ \mathcal{E}_k^{+} = \left\{ (u, v) \in \mathcal{E} \mid u, v \in \Omega_k^{+} \right\}.
\]
The oversampling procedure is a key step in defining local multiscale basis functions by localizing calculations and preserving the sparsity of the prolongation operator. 
In Figure \ref{fig:dd}, we plot subdomain $\Omega_k$ (red color) and the corresponding oversampled subdomains $\Omega_k^+$ for $\delta_H = 0.1$.

\subsection{Local spectral clustering and multiple coarse-scale variables}

Our approach for identifying multiple macroscopic variables in each local subdomain is based on local spectral clustering applied in each of the subdomains, which further partitions the DOFs into communities.  
For the local network $G_k$, we consider the localized signed graph Laplacian  \eqref{eq:laplacian} \cite{knyazev2018spectralpartitioningsignedgraphs, kunegis2010spectral}
\begin{equation}
L^{\Omega_k} = D^{\Omega_k} - W^{\Omega_k},  \quad 
D^{\Omega_k} = \text{diag}(d_1, \ldots,d_{n^{\Omega_k}}), \quad 
d_i = \sum_{j=1}^{n^{\Omega_k}} |w_{ij}|, 
\end{equation}
where $L^{\Omega_k}$, $W^{\Omega_k}$ and  $D^{\Omega_k}$ are the matrices defined on the sub-network $G^{\Omega_i}$ and  $n^{\Omega_k}$ is the  number of nodes in $\Omega_k$.  
Here $W^{\Omega_k}$ is the restriction of the global weight matrix $W$ to the subdomain $\Omega_k$.  
Then, spectral clustering is applied in each of the subgraphs $G_k$ (correspondingly, each of the subdomains $\Omega_k$) using the local generalized eigenvalue problem:
\begin{equation}
\label{sp1}
L^{\Omega_k} {\phi}^{\Omega_k}
= \lambda^{\Omega_k} D^{\Omega_K} {\phi}^{\Omega_k}.
\end{equation}
The local clusters, or aggregates, in MsGR are defined by first choosing $M_k$ eigenvectors $\phi_r^{\Omega_k}$ corresponding to the smallest eigenvalues $\lambda^{\Omega_k}_r$ with $r = 1,\ldots,M_k$ in each $G_k$. 
We note that for the sub-network $G_k$, in the special case that the edge weights are all positive, the local matrix $L^{\Omega_k}$ is positive semidefinite such that the eigenvalues satisfy $0 = \lambda^{\Omega_k}_1 \leq \lambda^{\Omega_k}_2 \leq \ldots \leq \lambda^{\Omega_k}_{M_k}$  of $L^{\Omega_k}$, i.e., they are all real and nonnegative; additionally, the smallest eigenvalue $\lambda^{\Omega_k}_1 = 0$ and its corresponding eigenvector is the constant \cite{gallier2016spectral}.  
Moreover, solving  the generalized eigenvalue problem involving $D$ 
is equivalent to using the normalized signed Laplacian in local spectral clustering. Theoretically, the use of $D$ in the generalized eigenproblem accounts for the scaling used in defining the norms in the convergence analysis of upscaling methods for highly heterogeneous problems (see the Appendix \ref{app1} for details).

\begin{figure}[h!]
\centering
\includegraphics[width=0.19\textwidth]{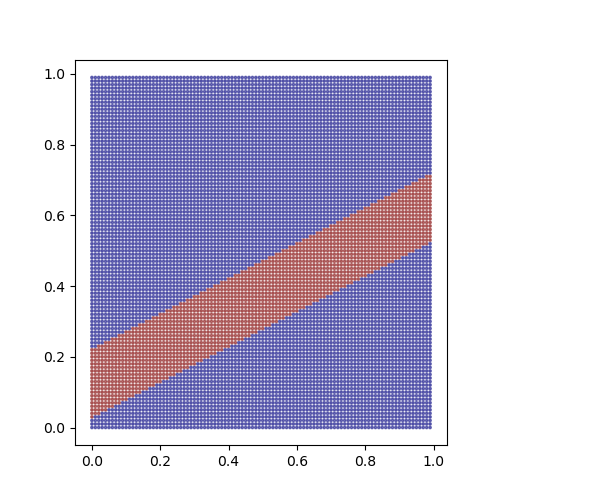}
\includegraphics[width=0.19\textwidth]{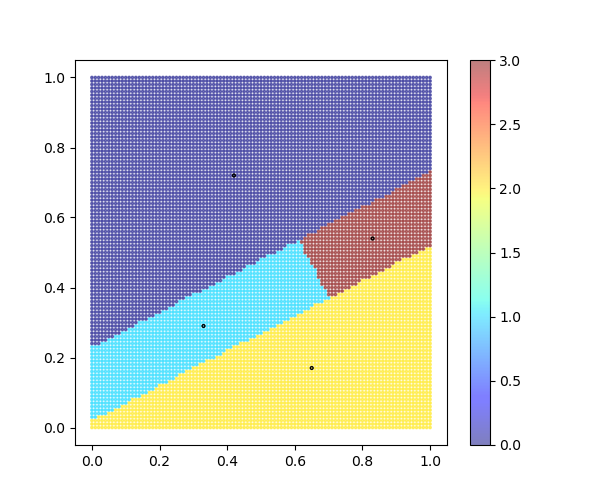}
\includegraphics[width=0.19\textwidth]{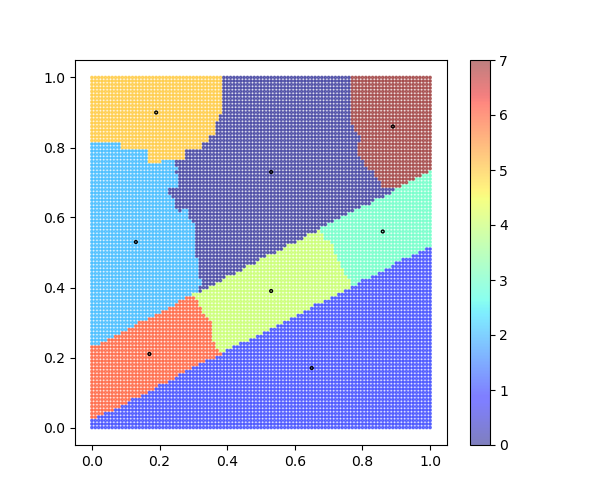}
\includegraphics[width=0.19\textwidth]{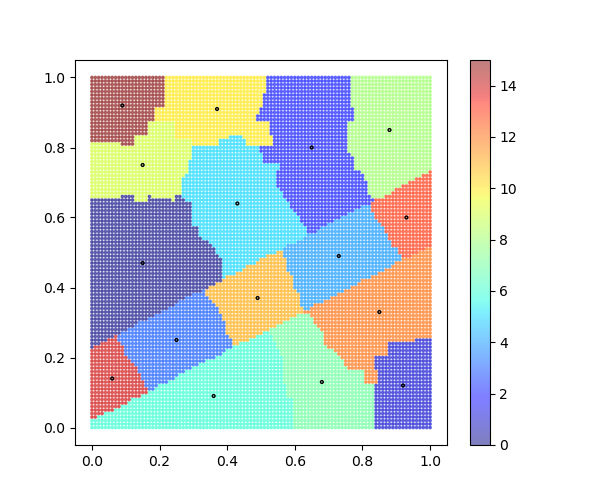}
\includegraphics[width=0.19\textwidth]{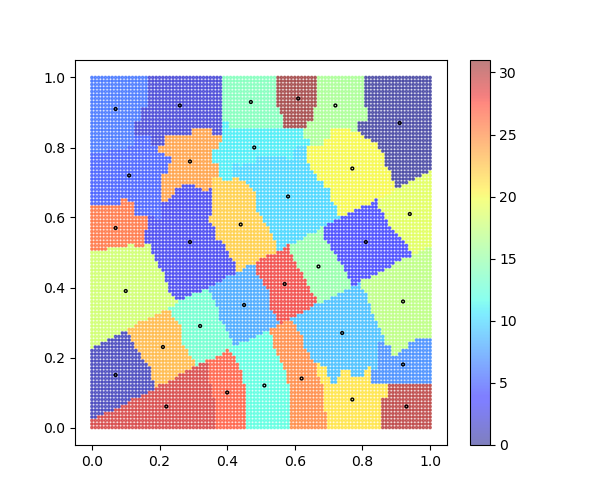}
\caption{Illustration of spectral clustering and centroids for high-contrast case, $K(x)$ (blue: $K = 1$, red: $K = 10^4$). The first plot displays the coefficient field $K(x)$. The second, third, fourth and fifth plots show the resulting partitions for 4, 8, 16 and 32 clusters.}
\label{fig:cfsplit1}
\end{figure}

\begin{figure}[h!]
\centering
\includegraphics[width=0.19\textwidth]{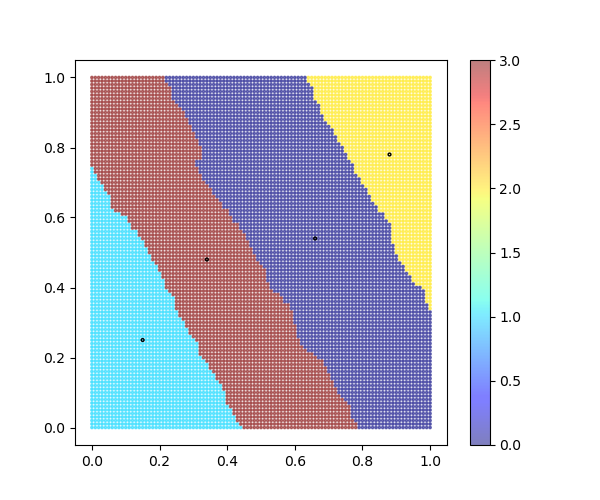}
\includegraphics[width=0.19\textwidth]{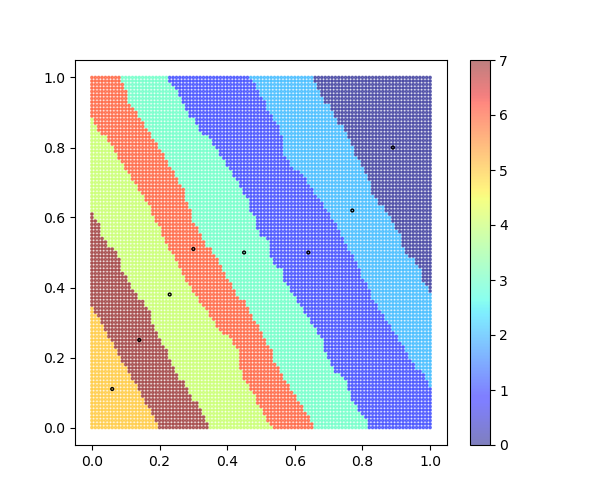}
\includegraphics[width=0.19\textwidth]{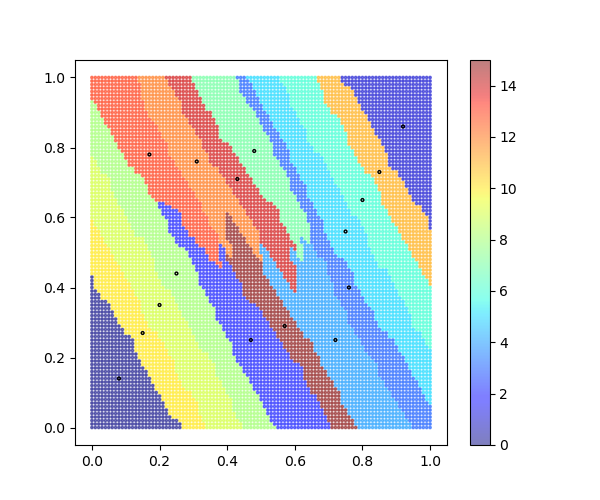}
\includegraphics[width=0.19\textwidth]{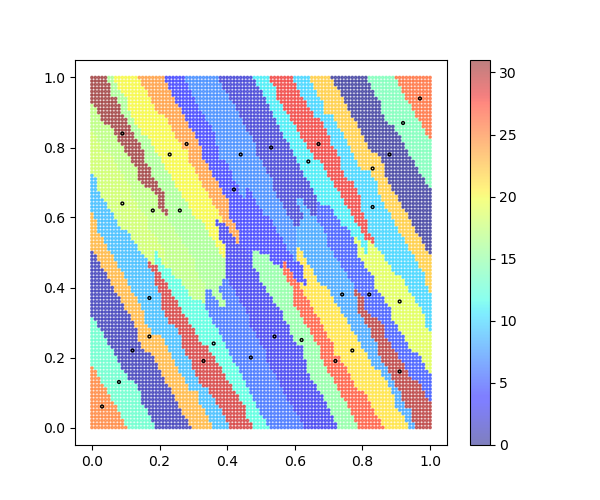}
\caption{Illustration of local clusters and their centroids for highly anisotropic case. Resulting partitions for 4, 8, 16 and 32 clusters (from left to right).}
\label{fig:cfsplit2}
\end{figure}

Next, we form a matrix of local eigenvectors $\Phi^{\Omega_k} =[\phi_1^{\Omega_k} \dots \phi_{M_k}^{\Omega_k}] \in \mathbb{R}^{n^{\Omega_k} \times M_k}$ and apply $k$-means~\cite{1056489, ng2001spectral, shi2000normalized} to the normalized rows of $\Phi^{\Omega_k}$ to partition the data into $M_k$ disjoint clusters (aggregates in AMG). 
{We denote the $r$-th local cluster in subdomain $\Omega_k$ by $\mathcal{A}_k^r$, where $k$ indexes the subdomain and $r$ indexes the cluster within that subdomain. These clusters satisfy}
$$
\bigcup_{r=1}^{M_k} \mathcal{A}_k^r = \Omega_k,  \quad k = 1, ..., N_{\Omega}. 
$$
Here, $N_{\Omega}$ is the number of subdomains generated by METIS and $M_k$ is the number of aggregates in each of the subdomains $\Omega_k$.
Once the subdomain nodes have been clustered in the spectral embedding space, we identify representative/coarse points as the centroids $\mu_j$ for each cluster. In this work, we define the centroid of each cluster as the geometrically closest data point (in the original space) to the mean of the cluster in the spectral space. 
We choose the closest point to the centroid $x_{c}^{\mathcal{A}_k^r}$ of cluster $\mathcal{A}_k^r$ such that
\[
x_{c}^{\mathcal{A}_k^r} = \arg\min_{x_i \in \mathcal{A}_k^r} \|x_i - \mu_r\|_2.
\]
This approach ensures that the selected centroid corresponds to an actual node of $\mathcal{A}_k^r$  
and provides a meaningful geometric interpretation of the cluster representatives. 

\begin{figure}[h!]
\centering
\includegraphics[width=0.8\textwidth]{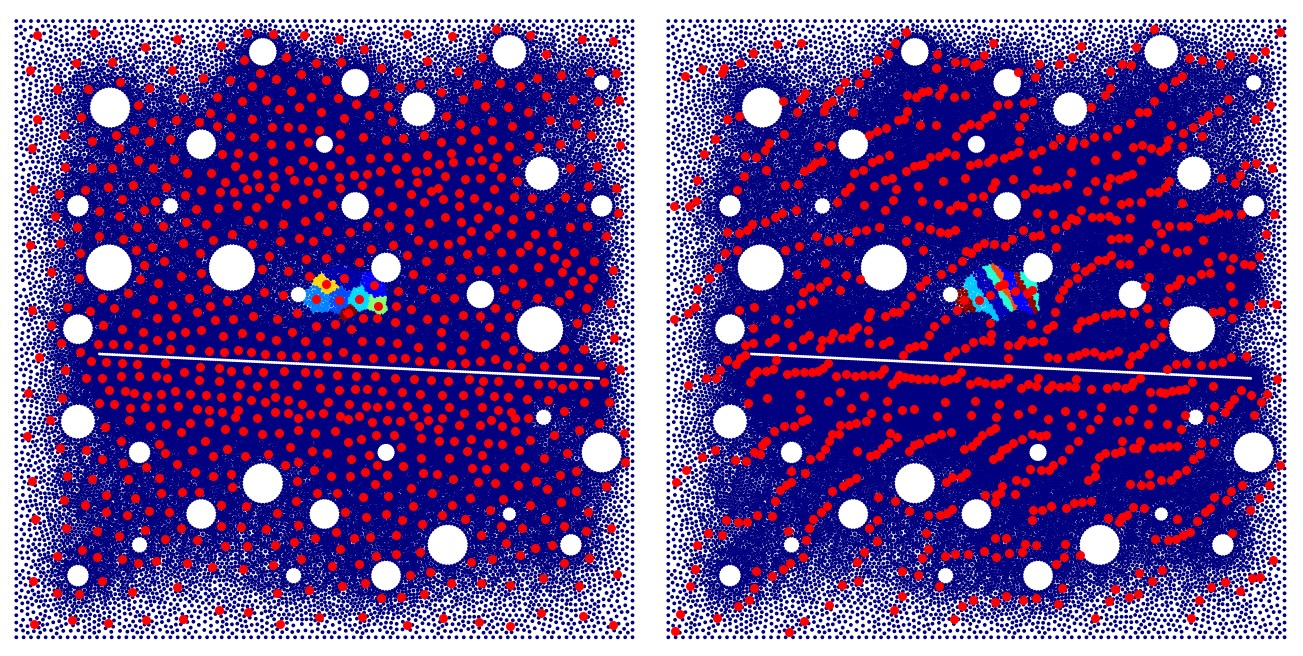}
\caption{Illustration of local clusters for heterogeneous perforated domain with centroids. 
Resulting partitions for 8 clusters in isotropic (left) and anisotropic cases (right). In the anisotropic case, note that centroids are narrowly spaced roughly orthogonal to the direction of anisotropy.}
\label{fig:cfsplit3}
\end{figure}

In Figure \ref{fig:cfsplit1}, the first plot corresponds to the heterogeneous properties of the scalar spatially varying diffusion coefficient $K(x)$. The second, third, fourth, and fifth plots illustrate the results of spectral clustering, corresponding to 4, 8, 16, and 32 clusters, respectively. The clusters are depicted with different colors and representative points (centroids). In Figure \ref{fig:cfsplit2}, we plot an illustration of spectral clustering for the fixed highly anisotropic case, $K = R^T D R$, where
\begin{equation}
\label{eq:aniso-k}
  R = \begin{bmatrix}
  \cos\theta & -\sin\theta \\
  \sin\theta & \cos\theta
  \end{bmatrix}, \quad
  D = \begin{bmatrix}
  d_1 & 0 \\
  0 & d_2
  \end{bmatrix},
\end{equation}
with $d_1 = 1$, $d_2 = 10^{-4}$ and $\theta = \pi/3$. 
Last, we consider a more general case involving a domain with heterogeneity, perforations, and potentially anisotropy. The heterogeneous domain consists of two subregions: the main domain (in blue) with permeability $K = 1$, and a highly diffusive channel (in red) with $K = 10^4$ (see the left plot in Figure \ref{fig:perf}). In addition, we illustrate the effect of anisotropy in the main domain on the formation of local clusters.
In Figure \ref{fig:cfsplit3}, we plot local clusters in subdomain $\Omega_k$, and global centroids for isotropic and anisotropic diffusion, with anisotropy as defined in \eqref{eq:aniso-k}. The plots illustrate how spectral clustering based on the given operator understands underlying connections and divides local subdomains into heterogeneity- or anisotropy-aligned clusters.
The eigenvectors of the weighted Laplacian encode structural information about the graph, and the spectral embedding reveals approximate block structures corresponding to weakly connected components. These blocks often reflect clusters with minimal interaction between them. Moreover, as seen in Figure 6, since the clusters appropriately align themselves with the heterogeneity or anisotropy, the centroids produce a coarse-fine splitting that accurately captures these features.

\subsection{Multiscale basis functions} 

We consider two approaches for constructing multiscale basis functions and the corresponding prolongation operators within a Galerkin coarsening approach. 
{The choice of these two approaches is intended to compare two standard coarse-variable viewpoints: a C-point-based viewpoint from classical multigrid/AMG and an averaging-based viewpoint from multiscale upscaling:}
\begin{itemize}
\item \textit{CF-approach:} Coarse nodes are defined by identifying centroids  and using node splitting into disjoint coarse and fine node groups, resulting in a CF-splitting defined in terms of the clustering.  Basis functions are computed column-wise as local approximations to ideal interpolation.  
\item \textit{MC-approach:} Coarse variables are associated with local clusters (aggregates of vertices) and basis functions are computed by solving local constrained energy minimization problems, effectively capturing the average behavior within each community of variables.
\end{itemize}
Overall, we consider global constructions of these approaches for constructing multiscale basis functions, as well as their corresponding localized versions.

\subsubsection*{CF-approach} 

In the \textit{CF-approach}, we partition the node set $\mathcal{V}$ into two disjoint subsets $\mathcal{V} = \mathcal{C} \cup \mathcal{F}$, $\mathcal{C} \cap \mathcal{F} = \emptyset$, where $\mathcal{C}$ denotes the set of coarse nodes and $\mathcal{F}$ the set of fine nodes.  
In AMG theory, such partitioning is called CF-splitting and plays a fundamental role in constructing the prolongation operator \cite{ruge1987algebraic, falgout2004generalizing}.
We start with an approach for deriving the coarse equations based on the global ideal interpolation operator. 
Given a CF-splitting, ideal interpolation reads
\begin{equation}
\label{eq:pideal}
P_{ideal} =
\begin{bmatrix}
\mathcal{W} \\
\mathcal{I}
\end{bmatrix}, \quad 
\mathcal{W} = -A_{\mathcal{F}\mathcal{F}}^{-1} A_{\mathcal{F}\mathcal{C}},
\quad 
A =
\begin{bmatrix}
A_{\mathcal{F}\mathcal{F}} & A_{\mathcal{F}\mathcal{C}} \\
A_{\mathcal{C}\mathcal{F}} & A_{\mathcal{C}\mathcal{C}}
\end{bmatrix},\quad 
\end{equation}
where $\mathcal{I}$ is the identity matrix interpolating the coarse variables and $\mathcal{W}$ represents the interpolation weights from the coarse to the fine variables. 

After coarsening, we have multiple coarse points (centroids $x_{c}^{\mathcal{A}_k^r}$) in each subdomain. Define a solution vector over the coarse space
$u_{\mathcal{C}} = (u_1^{\Omega_1}, \ldots, u_{M_1}^{\Omega_1}, \ldots, u_1^{\Omega_{N_{\Omega}}}, \ldots, u_{M_{N_{\Omega}}}^{\Omega_{N_{\Omega}}})$, 
where $u_r^{\Omega_k}$ represents $r$-th coarse variable in local domain $\Omega_k$, and $u_{\mathcal{C}} \in \mathbb{R}^{n_c}$ for $n_c = \sum_{k=1}^{N_{\Omega}} M_k$. Then let $ \psi_r^{\Omega_k} \in \mathbb{R}^n$ denote global multiscale basis functions defined as a column of the global ideal interpolation operator
\[
P_{glo} =
P_{ideal} =
\left[  
\psi_1^{\Omega_1} \ldots \psi_{M_1}^{\Omega_1} 
\ldots 
\psi_1^{\Omega_{N_{\Omega}}} \ldots \psi_{M_{N_{\Omega{}}}}^{\Omega_{N_{\Omega}}}
\right].
\]
Recall that, here, $N_{\Omega}$ is the number of local subdomains computed by the MsGR algorithm, and $M_k$ denotes the number of clusters that are computed in each of the subdomains. We can partition our multiscale basis functions across C- and F-points, $\psi_r^{\Omega_k} 
= ({\psi}_{r,f}^{\Omega_k}, {\psi}_{r,c}^{\Omega_k})$.
Note by the nature of the identity block in $P_{ideal}$, we have $\psi_{r,c}^{\Omega_k} =  \{ e_{r, j}^{\Omega_k} \}$ with $e_{r, \ell(j,b)}^{\Omega_k} = \delta_{ij, rb}$ and index $\ell(j,b)$ corresponding to the centroid  $x_{c}^{\mathcal{A}_j^b}$ of the local cluster $\mathcal{A}_j^b$, that is, $\psi_{r,c}^{\Omega_k}$ is $1$ at the target local cluster centroid $x_{c}^{\mathcal{A}_k^r}$ and $0$ elsewhere. Then, $\psi_r^{\Omega_k}  = (A_{\mathcal{F}\mathcal{F}}^{-1} A_{\mathcal{F}\mathcal{C}}{\psi}_{r,c}^{\Omega_k}, {\psi}_{r,c}^{\Omega_k})$.
From here we can approximate the fine-grid solution as interpolated from coarse multiscale basis functions via
\[
u \approx P_{ideal}u_{\mathcal{C}} = \sum_{k=1}^{N_{\Omega}} \sum_{r=1}^{M_k} 
u_r^{\Omega_k} \psi_r^{\Omega_k}.
\]
Note, there exists coarse $u_{\cal C}$ such that this interpolation yields the exact fine-grid solution by nature of linear reduction.

We can localize the above global basis construction and perform calculations in oversampled domains, $\Omega_k^+$.
We define basis functions as follows:
\[
\tilde{\psi}_r^{\Omega_k} 
= (\tilde{\psi}_{r,f}^{\Omega_k}, \tilde{\psi}_{r,c}^{\Omega_k}) 
= (\mathcal{W}^{\Omega_k^+} \tilde{\psi}_{r,c}^{\Omega_k}, \tilde{\psi}_{r,c}^{\Omega_k}), \quad 
\mathcal{W}^{\Omega_k^+} = - \Big(A^{\Omega_k^+}_{\mathcal{F}\mathcal{F}} \Big)^{-1} A_{\mathcal{F}\mathcal{C}}^{\Omega_k^+},
\quad 
A^{\Omega_k^+} =
\begin{bmatrix}
A^{\Omega_k^+}_{\mathcal{F}\mathcal{F}} & 
A^{\Omega_k^+}_{\mathcal{F}\mathcal{C}} \\
A^{\Omega_k^+}_{\mathcal{C}\mathcal{F}} & 
A^{\Omega_k^+}_{\mathcal{C}\mathcal{C}}
\end{bmatrix}, 
\]
where $\tilde{\psi}_{r,c}^{\Omega_k}$ is the vector with zeros in all elements except continuum $r$ in target aggregate $\mathcal{A}_k$. The operator $A^{\Omega_k^+}$ is the block of the global operator $A$ corresponding to the oversampled domain of the $k$th aggregate, $\Omega_k^+$.     
The local construction of prolongation is then given by
\[
P_{loc} =
\left[  
\tilde{\psi}_1^{\Omega_1} \ldots \tilde{\psi}_{M_1}^{\Omega_1} 
\ldots 
\tilde{\psi}_1^{\Omega_{N_{\Omega}}} \ldots \tilde{\psi}_{M_{N_{\Omega}}}^{\Omega_{N_{\Omega}}}
\right],
\]
where $\tilde{\psi}_r^{\Omega_k}$ are the localized basis functions mapped to the global indexing.  This procedure is analogous to the approach used in constructing $\ell$AIR restriction and interpolation~\cite{AIR,manteuffel2019nonsymmetric,constrainedAIR}, where ideal interpolation (or restriction) is enforced exactly over local overlapping subdomains, and is also used in energy minimization techniques for balancing domain decomposition methods~\cite{https://doi.org/10.1002/nla.341}. Using MsGR coarsening with this interpolation can be seen as a refinement to both these AMG and DD approaches.

\subsubsection*{MC-approach} 

In the \textit{MC-approach}, coarse variables identified with clusters are considered as a community  of nodes that exhibit similar behavior and are represented in an averaged manner, as in unsmoothed aggregation AMG~\cite{https://doi.org/10.1002/nla.741}.  The method is motivated by the Nonlocal Multicontinua Upscaling (NLMC) method \cite{chung2018non, vasilyeva2019nonlocal, vasilyeva2019upscaling}, where each constraint corresponds to a specific continuum. However, in NLMC, these constraints must be predefined prior to computation, which is only feasible in simple settings where the continua can be manually identified. 

To construct multiscale basis functions associated with subdomain $\Omega_k$ in the \textit{MC-approach}, we begin with a globally constrained energy minimization problem. Let $\mathcal{A}_k^{r}$ denote the aggregate associated with the local clusters $r = 1, \ldots, M_k$, and assume that $\cup_{r=1}^{M_k} \mathcal{A}_k^{r} = \Omega_k$.  Hence, we have an aggregation for each subdomain.  
Then, for a fixed pair $(k, r)$, the basis function ${\psi}_r^{\Omega_k}$ is obtained by solving:
\begin{equation} \label{eq:min-prob}
\min_{\psi} \ \frac{1}{2} \psi^T A \psi \quad \text{subject to} \quad S^{\mathcal{A}_j^{b}} \psi = \delta_{kj} \delta_{rb}, \quad \forall j = 1,\ldots, N_{\Omega}, \quad b = 1,\ldots, M_j,
\end{equation}
where $S^{\mathcal{A}_j^{b}}$ represents the constraint operator associated with the aggregate $\mathcal{A}_j^{b}$. The Kronecker delta $\delta_{kj}$ and $\delta_{rb}$ ensure the physical meaning of the basis function with respect to the $r$-th constraint in $\Omega_k$.  
The operator $S^{\mathcal{A}_j^{b}}$ represents the local constraint operator associated with the aggregate $\mathcal{A}_j^{b}$. It computes the mean value of a fine-scale function over the aggregate and enforces the continuum separation conditions. The discrete form of $S^{\mathcal{A}_j^{b}}$ is given as a row vector with entries 1 or 0
\[
S^{\mathcal{A}_j^{b}}  = \{s_i\}, 
\quad
s_i = 
\begin{cases}
1/n^{\mathcal{A}_j^{b}}, & \text{if } x_i \in \mathcal{A}_j^{b}, \\ 
0, & \text{otherwise.}
\end{cases},
\]
where $n^{\mathcal{A}_j^{b}} = |\mathcal{A}_j^{b}|$. 
The constraint equation ensures that the basis function associated with the aggregate  $\mathcal{A}_k^{r}$ has an average value of one over its own aggregate $\mathcal{A}_k^r$, and zero over all other aggregates $\mathcal{A}_j^{b}$ within the global domain (graph).

To solve the constrained minimization problem, we apply the method of Lagrange multipliers, leading to the following saddle-point system:
\begin{equation}
\label{ms-cem-gl}
\begin{split}
& A {\psi}_r^{\Omega_k}
+ \sum_{j=1}^{N_{\Omega}} \sum_{b=1}^{M_j}  (S^{\mathcal{A}_j^b})^T \gamma_{b, j} = 0,\\
& S^{\mathcal{A}_j^{b}} {\psi}_r^{\Omega_k}  = \delta_{kj} \delta_{rb}, \quad \forall \mathcal{A}_j^{b}, 
\end{split}
\end{equation}
where  $\gamma_{b,j}$ are the Lagrange multipliers associated with the constraints imposed on the aggregate $\mathcal{A}_j^b$.

For a localized approach,  we construct a set of basis functions $\tilde{\psi}_r^{\Omega_k} \in \Omega_k^+$ by solving the local problems with zero Dirichlet boundary conditions subject to continuum separation constraints. 
We solve the following constrained cell problems formulated using  Lagrange multipliers 
\begin{equation}
\label{ms-a2}
\begin{split}
& A^{\Omega_k^+} \tilde{\psi}_r^{\Omega_k}
+ \sum_{\Omega_j \subset \Omega_k^+} \sum_{\mathcal{A}_j^{b} \subset \Omega_j}  (S^{\mathcal{A}_j^b})^T \gamma_{b, j} = 0,\\
& S^{\mathcal{A}_j^{b}} \tilde{\psi}_j^{\Omega_b}  = \delta_{kj} \delta_{rb}, \quad \forall \mathcal{A}_j^b \subset \Omega_j \subset  \Omega_k^+.
\end{split}
\end{equation}
The operator $S^{\mathcal{A}_j^{b}}$ represents the local constraint operator associated 
with the aggregate $\mathcal{A}_j^{b}$. 
The constraints give meaning to the coarse scale solution and produce a mean value of one on the coarse aggregate $\mathcal{A}_k^r$ for the current community and zero on all other aggregates within $\Omega_k^+$.

Finally, we define prolongation operators related to global and localized versions
\[
P_{glo} =
\left[  
\psi_1^{\Omega_1} \ldots \psi_{M_1}^{\Omega_1} 
\ldots 
\psi_1^{\Omega_{N_{\Omega}}} \ldots \psi_{M_{N_{\Omega}}}^{\Omega_{N_{\Omega}}}
\right], \quad 
P_{loc} =
\left[  
\tilde{\psi}_1^{\Omega_1} \ldots \tilde{\psi}_{M_1}^{\Omega_1} 
\ldots 
\tilde{\psi}_1^{\Omega_{N_{\Omega}}} \ldots \tilde{\psi}_{M_{N_{\Omega}}}^{\Omega_{N_{\Omega}}}
\right],
\]
where ${\psi}_r^{\Omega_k}$ are the global basis functions and and $\tilde{\psi}_r^{\Omega_k}$ the localized multiscale basis functions mapped to the global indexing. 

\subsection{Coarse-scale system}

To define a coarse grid approximation of steady-state system $Au = f$, we use a prolongation matrix $P \in \mathbb{R}^{n_c \times n}$ and form a coarse-grained system as follows
\begin{equation}
\label{eq:coarse}
A_c u_c = f_c, \quad A_c = R A P \in \mathbb{R}^{n_c \times n_c}, \quad f_c = R f \in \mathbb{R}^{n_c},
\end{equation}
where we set $R = P^T \in \mathbb{R}^{n \times n_c}$ for the restriction operator for symmetric problems. 
Moreover, we can reconstruct the fine-scale solution
\[
    u_{ms} \coloneqq P u_c  \in \mathbb{R}^{n}.
\]
We note that the error depends on the partitioning into subdomains $\Omega_k$  ($H$) and intra-cluster heterogeneity $C_{{ratio}}$ (contrast ratios within local cluster)
\[
\|u - u_{ms}\|_A = \mathcal{O} \left(H C_{\text{ratio}}^{1/2}\right).
\]
where $H = \max_{k,r} H_{k,r}$, $H_{k,r} = \text{diam}(\mathcal{A}_{k}^r)$ and $C_{{ratio}} = \max_{k,r} C_{{ratio}}^{k,r}$, $C_{{ratio}}^{k,r} = \overline{w}^{r,k}/\underline{w}^{r,k}$  {is defined over nonzero intra-cluster connections}  ($\underline{w}^{k,r} \leq w_{ij} \leq \overline{w}^{k,r}$ and $w_{ij}$ is the connection weight of edge $e_{ij} \in \mathcal{G}_k$).
{Here, zero weights indicate absent edges and are not included in this ratio.}  
Here $u$ is the solution of fine-scale system \eqref{eq:sys-f}, $u_{ms} = P u_c$ is the multiscale interpolated solution of \eqref{eq:coarse}, and $\|v\|^2_{A} = v^T A v$  (see Appendix \ref{app1} for details). 

\section{Numerical results}

To demonstrate the effectiveness of the proposed local spectral clustering approach, we examine three test problems: \textit{(Test 1)} an elliptic equation discretized by finite elements on a heterogeneous perforated domain; \textit{(Test 2)} a highly anisotropic heat flow problem; and \textit{(Test 3)} a high-contrast pore network model.
We compare a multiscale solution $u_{ms}$ with reference solution in  $L_2$ ($\|u\|^2 = v^T v$) and energy ($\|v\|^2_{A} = v^T A v$) norms
\[
e_1 = \frac{\|u-u_{ms}\|}{\|u\|} \times 100 \%, \quad 
e_2 = \frac{\|u-u_{ms}\|_A}{\|u\|_A} \times 100 \%, 
\]
where the relative error is presented in percentage. As a reference solution, we use the fine-scale solution. 
{Thus, for the pore-network examples, the reported errors quantify the accuracy of the coarse-scale reduction relative to the fine-scale pore-network system, not the uncertainty or modeling error of the pore-network model itself.} 
Implementation is performed using the Python with the SciPy sparse library to represent and solve the resulting system of linear equations. We used PyMetis (Python wrapper for the METIS library \cite{karypis1997metis}) for graph partitioning. 

\subsection{Test 1: Elliptic problem in heterogeneous perforated domain}

We consider a heterogeneous perforated domain $\Omega = [0,1]^2$ with 35 circular perforations and one thin long perforation in the middle of the domain (see Figure \ref{fig:perf}).
We have heterogeneous coefficient $K(x)$ defined in $\Omega = \Omega_1 \cup \Omega_2$,  where we set $K(x) = K_1$ in $\Omega_1$ (main domain) and $K(x) = K_2$ in $\Omega_2$ (channels domain) (see Figure \ref{fig:perf}, blue and red color represent $\Omega_1$ and $\Omega_2$).

\begin{figure}[h!]
\centering
\begin{minipage}[t]{0.95\textwidth} 
\centering
\includegraphics[width=1\textwidth]{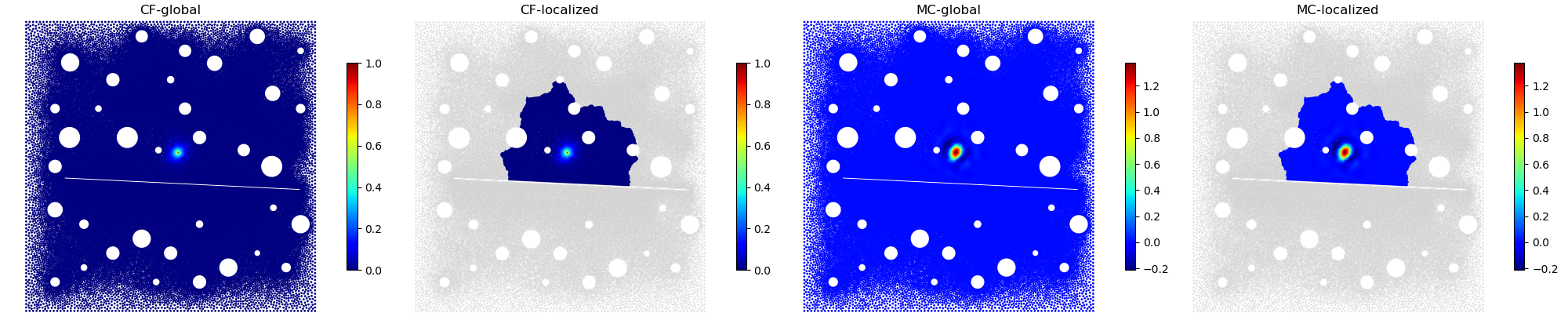}
\subcaption{Test 1a (isotropic)}
\end{minipage}
\begin{minipage}[t]{0.95\textwidth} 
\centering
\includegraphics[width=1\textwidth]{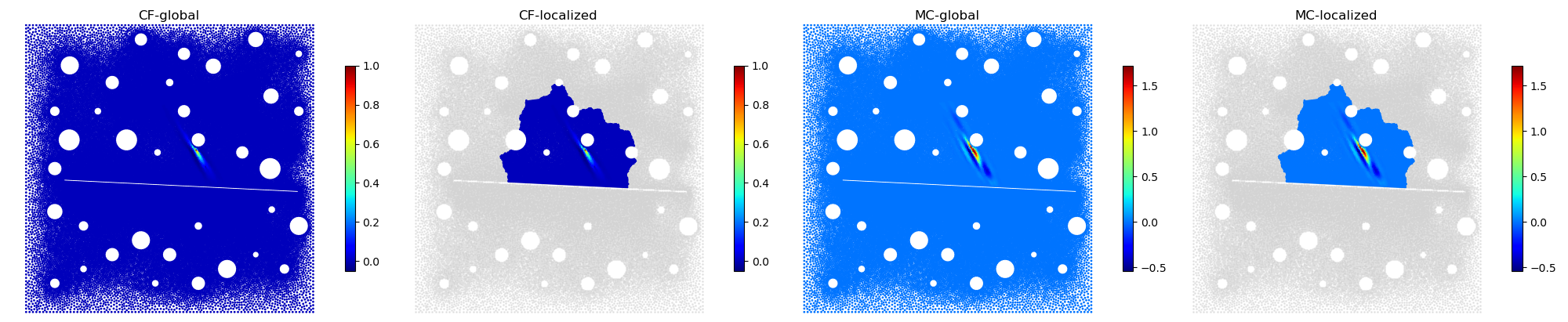}
\subcaption{Test 1b (anisotropic)}
\end{minipage}
\caption{Tests 1a and 1b. Multiscale basis functions (global and local) for CF- and MC-approaches  
{with $N_{\Omega}=100$ subdomains and $M=8$ local clusters. 
Top row: Test 1a with isotropic background coefficient $K_1=1$. Bottom row: Test 1b with anisotropic background tensor $K_1=K_{aniso}$. 
In each row, the columns compare CF-global, CF-localized, MC-global, and MC-localized basis functions (gray color indicate points outside the localized support).}}
\label{fig:test1-msbasis}
\end{figure}

\begin{figure}[h!]
\centering
\begin{minipage}[t]{0.45\textwidth} 
\centering
\includegraphics[width=1\textwidth]{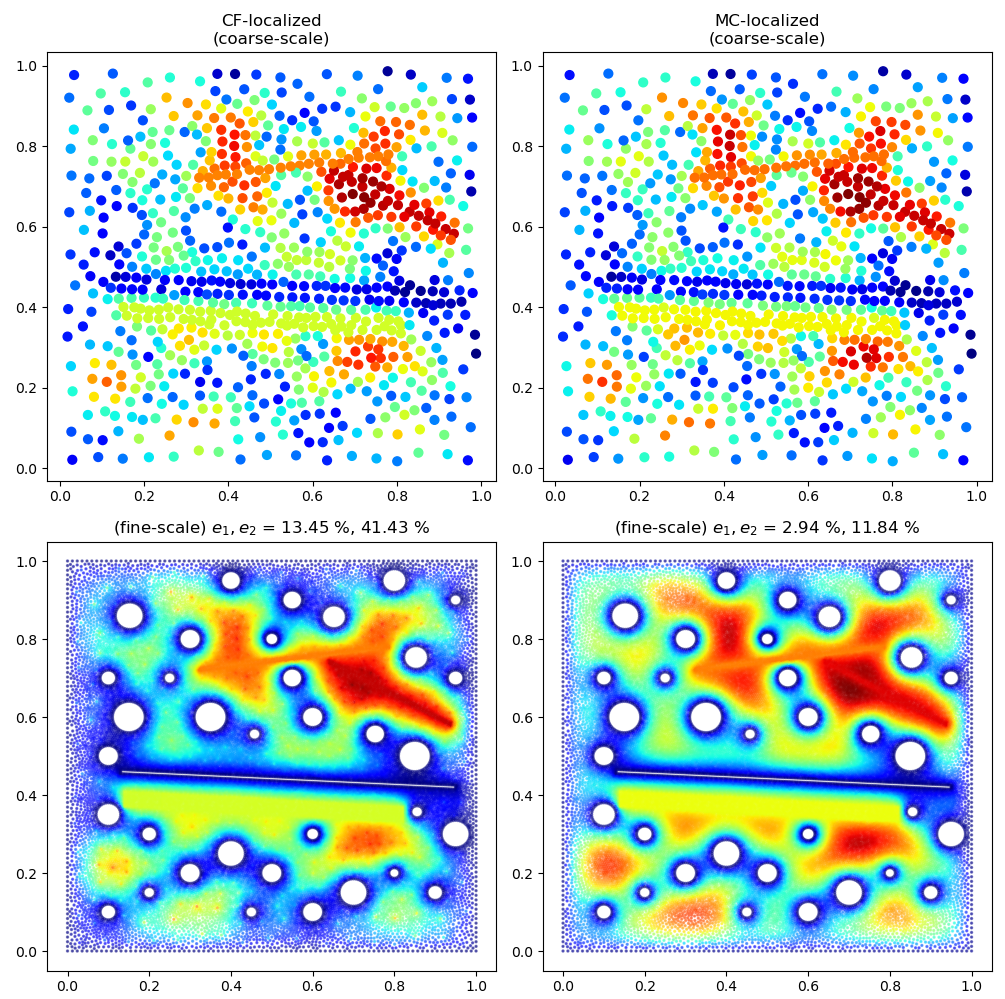}
\subcaption{Test 1a (isotropic)}
\end{minipage}
\ \ \ \
\begin{minipage}[t]{0.45\textwidth} 
\centering
\includegraphics[width=1\textwidth]{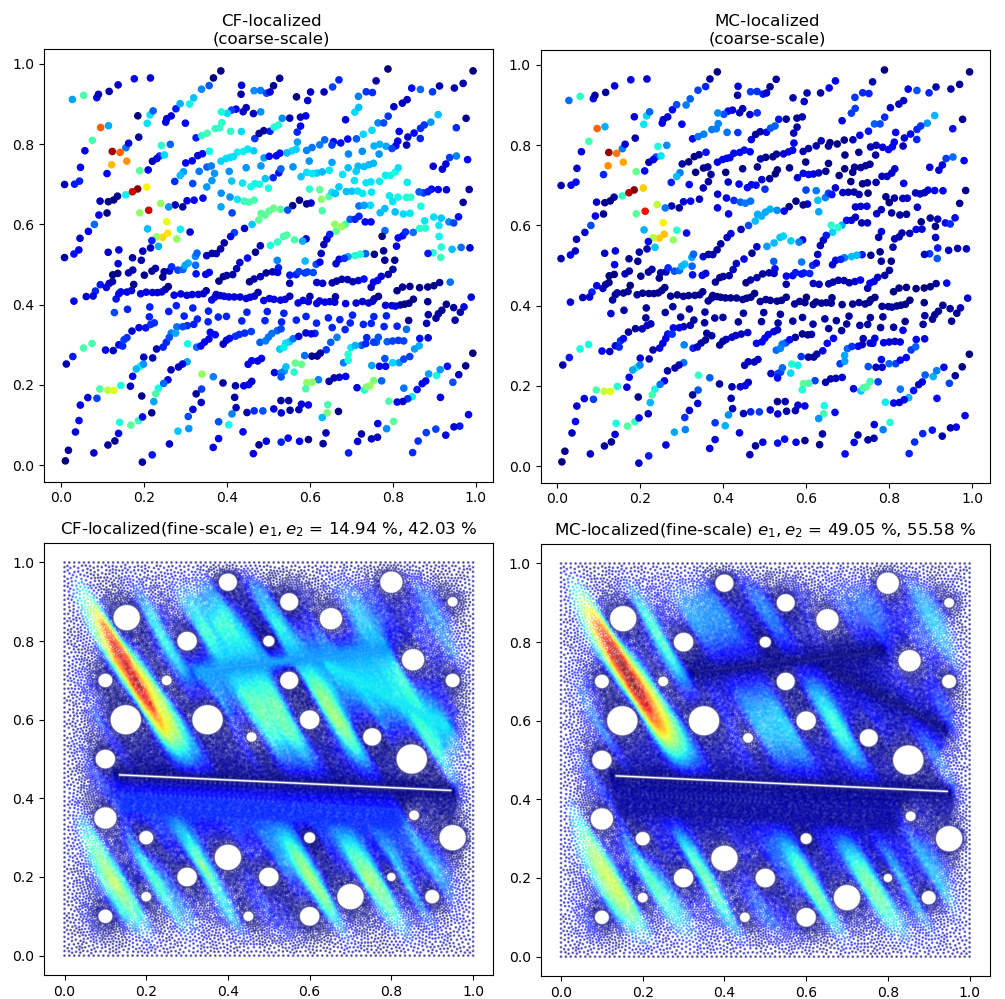}
\subcaption{Test 1b (anisotropic)}
\end{minipage}
\caption{Tests 1a and 1b. Multiscale solution: coarse-scale solution $u_c$ (first row), projected to fine-scale $u_{ms}$ (second row) and after three colored local correction iterations  (third row).  $DOF_c = 800$ (coarse-scale, $N_{\Omega}=100$ and $M=8$) and $DOF = 41,231$ (fine-scale)}
\label{fig:test1}
\end{figure}

\begin{table}[ht!]
\centering
\begin{tabular}{|cc||cc|cc||cc|cc||}
\hline
\multirow{4}{*}{M} & 
& \multicolumn{4}{c||}{$N_{\Omega} = 25$} 
& \multicolumn{4}{c||}{$N_{\Omega} = 100$}\\ 
 & 
& \multicolumn{2}{c}{CF} & \multicolumn{2}{c||}{MC} 
& \multicolumn{2}{c}{CF} & \multicolumn{2}{c||}{MC}  \\
& 
& $e_{1}$ & $e_{2}$  & $e_{1}$ & $e_{2}$ 
& $e_{1}$ & $e_{2}$  & $e_{1}$ & $e_{2}$ \\
\hline
\multicolumn{10}{|c|}{Test 1a}\\
\hline
1 & glo & 73.36 & 86.66 & 9.76 & 21.58 & 44.66 & 71.17 & 3.55 & 11.95 \\
 & loc & 93.12 & 96.01 & 51.61 & 54.61 & 84.51 & 89.31 & 62.45 & 57.85 \\
2 & glo & 49.82 & 75.24 & 6.55 & 17.30 & 32.89 & 62.13 & 2.32 & 9.29 \\
 & loc & 90.65 & 93.88 & 58.13 & 57.00 & 54.85 & 73.84 & 53.36 & 53.26 \\
4 & glo & 42.84 & 70.00 & 3.81 & 12.62 & 21.94 & 51.84 & 1.41 & 6.33 \\
 & loc & 62.99 & 79.14 & 51.83 & 52.32 & 22.24 & 52.22 & 7.66 & 20.33 \\
8 & glo & 33.32 & 62.46 & 2.06 & 8.45 & 13.41 & 31.37 & 1.15 & 4.81 \\
 & loc & 33.49 & 62.66 & 32.53 & 35.52 & 13.45 & 41.43 & 2.94 & 11.84 \\
16 & glo & 23.03 & 52.78 & 1.32 & 5.91 & 7.14 & 30.65 & 1.11 & 4.34 \\
 & loc & 23.06 & 52.82 & 4.59 & 15.73 & 7.14 & 30.66 & 1.11 & 4.34 \\
32 & glo & 13.95 & 41.89 & 1.16 & 4.79 & 3.49 & 21.63 & 1.04 & 3.99 \\
 & loc & 13.95 & 41.90 & 5.14 & 13.94 & 3.49 & 21.63 & 1.04 & 3.99 \\
\hline
\multicolumn{10}{|c|}{Test 1b}\\
\hline
1 & glo & 78.29 & 88.08 & 6.56 & 16.93 & 53.04 & 72.65 & 2.79 & 9.67 \\
 & loc & 91.02 & 96.12 & 38.97 & 52.03 & 88.37 & 93.15 & 51.69 & 59.27 \\
2 & glo & 70.59 & 81.76 & 5.24 & 14.51 & 40.35 & 63.71 & 2.12 & 8.27 \\
 & loc & 88.84 & 94.15 & 41.88 & 53.38 & 81.03 & 88.12 & 51.27 & 58.32 \\
4 & glo & 46.51 & 69.45 & 3.73 & 12.20 & 23.80 & 51.73 & 1.35 & 6.66 \\
 & loc & 69.45 & 84.67 & 42.01 & 52.81 & 59.05 & 75.33 & 50.99 & 57.59 \\
8 & glo & 36.80 & 61.89 & 2.57 & 10.39 & 11.74 & 37.93 & 0.81 & 5.62 \\
 & loc & 42.93 & 67.85 & 45.92 & 54.65 & 14.94 & 42.03 & 49.05 & 55.58 \\
16 & glo & 21.12 & 49.11 & 1.97 & 9.02 & 5.48 & 26.81 & 0.41 & 3.88 \\
 & loc & 24.67 & 52.83 & 46.07 & 52.68 & 5.62 & 27.10 & 20.31 & 29.02 \\
32 & glo & 10.95 & 36.64 & 0.91 & 6.07 & 2.86 & 19.28 & 0.33 & 3.33 \\
 & loc & 11.00 & 36.72 & 29.71 & 36.34 & 2.86 & 19.29 & 0.32 & 3.92 \\
\hline
\end{tabular}
\caption{Tests 1a and 1b. MsGR with global and localized multiscale basis functions}
\label{table:off-t1}
\end{table}

Geometry and triangulation are constructed using the Gmsh library \cite{geuzaine2009gmsh}. We use an unstructured locally refined mesh with 80,746 cells, 122,012 edges, and 41,231 nodes. 
We solve the diffusion problem \eqref{eq:ell} with zero Dirichlet boundary conditions on all boundaries and set the source term $q=1$. 
We consider isotropic and anisotropic backgrounds ({$K_1$}). In the isotropic case (\textit{Test 1a}), we set $K_1 = K_{iso} = 1$. In anisotropic case (\textit{Test 1b}), the diffusion tensor $K_1=K_{aniso} = (R_D)^T D R_D$ is defined by \eqref{eq:aniso-k} with $d_{1} = 1$ and $d_{2} = 10^{-4}$ and $\theta = \pi/3$. 

Figure \ref{fig:test1-msbasis} shows the multiscale basis functions obtained using the CF and MC approaches with both global and localized constructions. 
{The two rows separate the isotropic and anisotropic choices of the background coefficient $K_1$ and the four panels in each row compare the CF and MC constructions with global and localized supports. We see how localization restricts the support of the basis functions and how the anisotropic coefficient changes their orientation.} 
Localization is performed with $\delta_H=0.1$ (see Figure \ref{fig:dd}). 
In Figure \ref{fig:test1}, we present the coarse- and fine-scale solutions obtained with localized CF and MC approaches. The first row illustrates the solutions at the coarse degrees of freedom, where the values are shown at the cluster centroids, demonstrating that the coarse models remain physically meaningful. The corresponding coarse graph solutions are then projected onto the fine-scale resolution using the prolongation matrix $P$, and the results are plotted in the second row. Additionally, the errors $e_1$ and $e_2$ are reported.

In Table \ref{table:off-t1}, we present relative errors in \%. Localization is performed with $\delta_H=0.1$ (see Figure \ref{fig:dd}). We observe the error decay in $L_2$ and energy norms ($e_1$ and $e_2$) with respect to the number of basis functions $M$ corresponding to the number of local clusters, for two settings of $N_{\Omega} = 25$ and $N_{\Omega} = 100$ (number of subdomains). Both global (glo) and local (loc) strategies in basis construction are tested for the CF-approach and the MC-approach. As the number of local clusters $M$ increases, the error consistently decreases across all methods, with global construction performing better than the localized version. The larger number of local domains $N_{\Omega} = 100$ yields significantly improved accuracy, especially for the CF-approach. We see that with larger $M$, we obtain better localization for basis construction.

\subsection{Test 2: Anisotropic heat diffusion}

We consider anisotropic diffusion in a domain $\Omega = [0,1]^2$ for temperature $u = u(x)$
\[
- \nabla \cdot   q = f, \quad 
q  = k_{\parallel} \nabla_{\parallel} u + k_{\perp} \nabla_{\perp} u, 
\]
with $\nabla_{\parallel} (\cdot) = (b \cdot \nabla (\cdot)) b$ and $\nabla_{\perp} = \nabla - \nabla_{\parallel}$, $b=B/|B|$, where $B$ is the given magnetic field, and $k_{\parallel}$ and $k_{\perp}$ are the heat conductivity parallel and perpendicular to normalized magnetic field lines $b$, respectively.

\begin{figure}[h!]
\centering
\begin{minipage}[t]{0.95\textwidth} 
\centering
\includegraphics[width=1\textwidth]{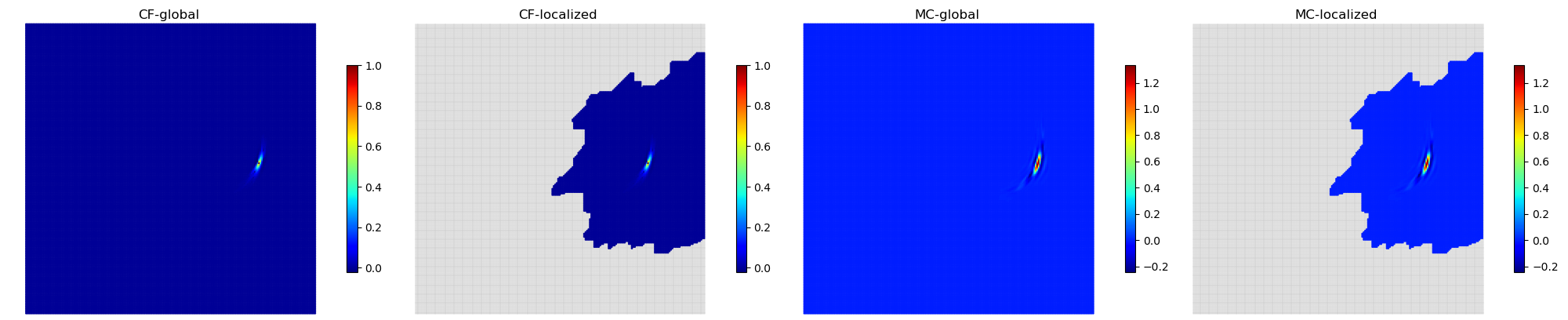}
\subcaption{Test 2a}
\end{minipage}
\begin{minipage}[t]{0.95\textwidth} 
\centering
\includegraphics[width=1\textwidth]{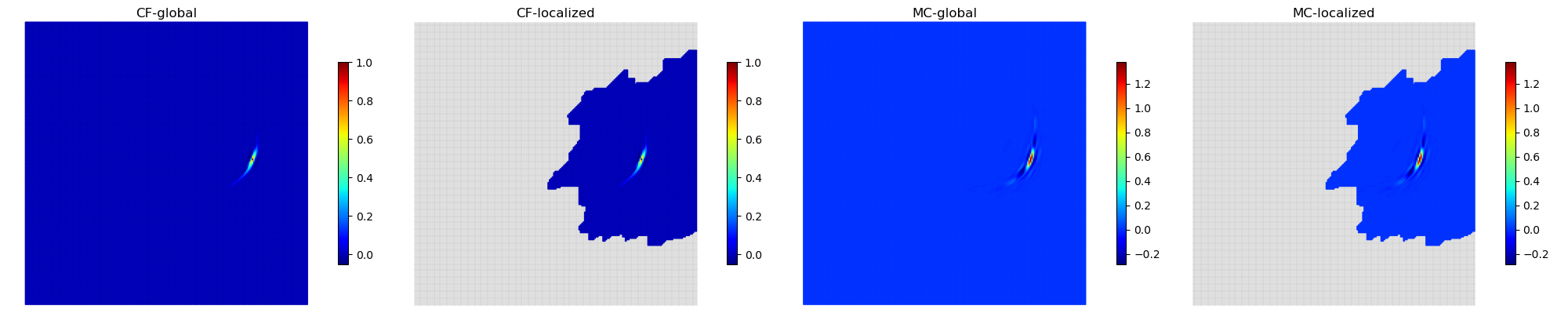}
\subcaption{Test 2b}
\end{minipage}
\caption{Tests 2a and 2b. Multiscale basis functions (global and local) for CF- and MC-approaches  
{with $N_{\Omega}=100$ subdomains and $M=32$ local clusters. 
Top row: Test 2a with moderate anisotropy, $k_{\perp}=10^{-1}$. 
Bottom row: Test 2b with stronger anisotropy, $k_{\perp}=10^{-3}$. 
In each row, the columns compare CF-global, CF-localized, MC-global, and MC-localized basis functions (gray color indicate points outside the localized support).}}
\label{fig:test2-msbasis}
\end{figure}

\begin{figure}[h!]
\centering
\begin{minipage}[t]{0.45\textwidth} 
\centering
\includegraphics[width=1\textwidth]{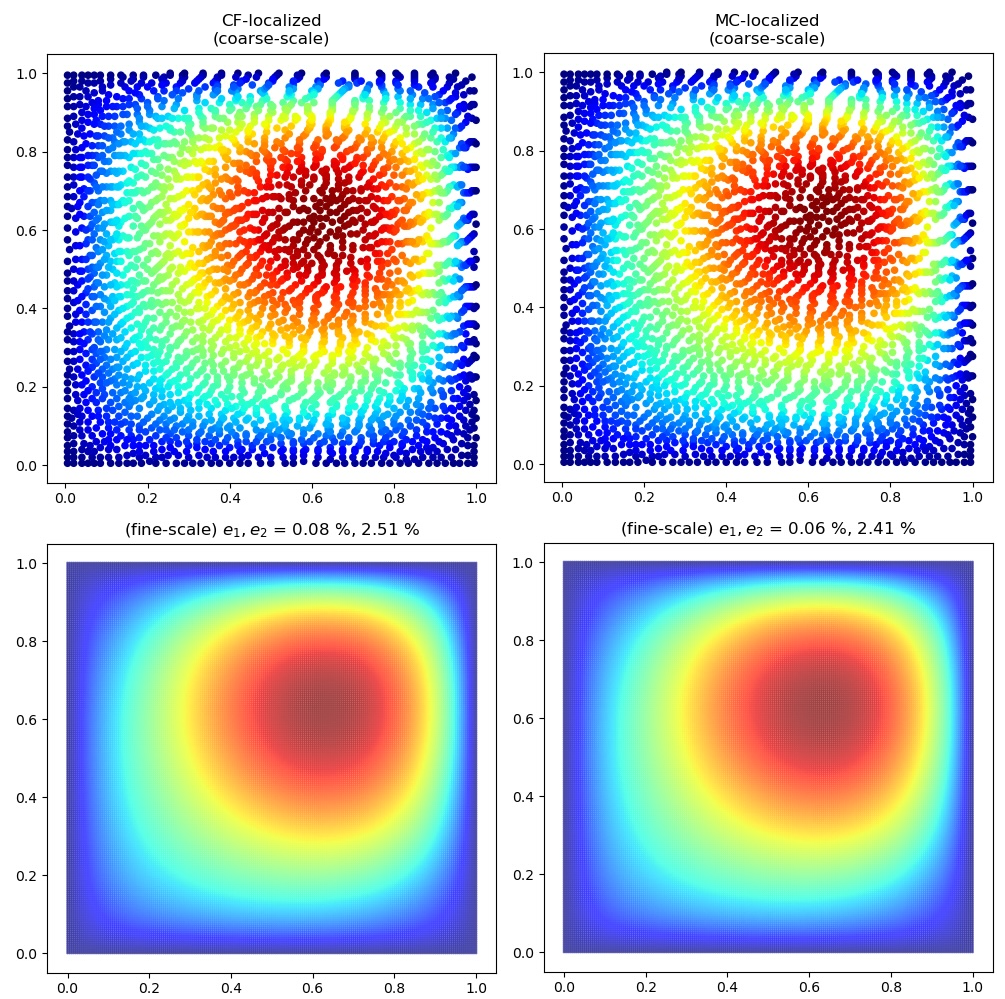}
\subcaption{Test 2a}
\end{minipage}
\ \ \ \
\begin{minipage}[t]{0.45\textwidth} 
\centering
\includegraphics[width=1\textwidth]{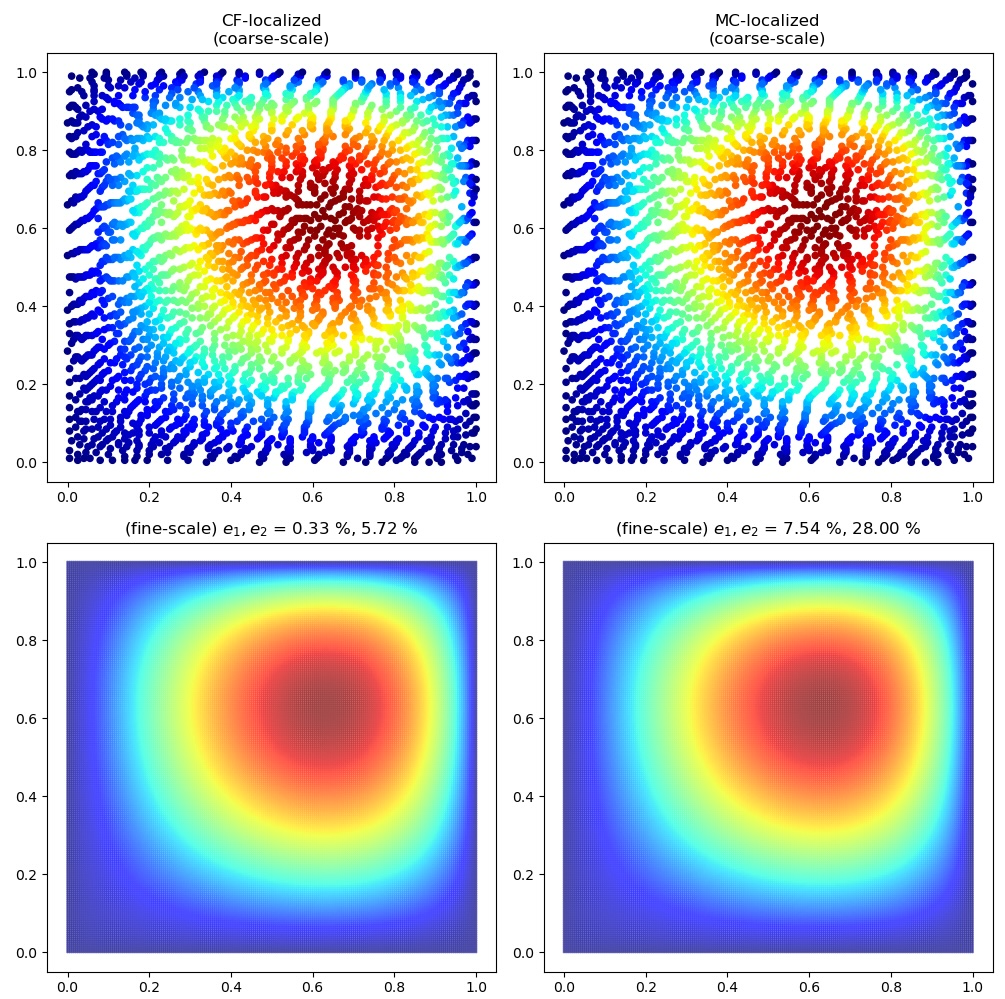}
\subcaption{Test 2b}
\end{minipage}
\caption{Tests 2a and 2b. Multiscale solution: coarse-scale solution $u_c$ (first row) and projected to fine-scale $u_{ms}$ (second row). 
$DOF_c = 3,200$ (coarse-scale, $N_{\Omega}=100$ and $M=32$) and $DOF = 40,401$ (fine-scale)}
\label{fig:test2}
\end{figure}

\begin{table}[ht!]
\centering
\begin{tabular}{|cc||cc|cc||cc|cc||}
\hline
\multirow{4}{*}{M} & 
& \multicolumn{4}{c||}{$N_{\Omega} = 25$} 
& \multicolumn{4}{c||}{$N_{\Omega} = 100$}\\ 
 & 
& \multicolumn{2}{c}{CF} & \multicolumn{2}{c||}{MC}
& \multicolumn{2}{c}{CF} & \multicolumn{2}{c||}{MC}   \\
& 
& $e_{1}$ & $e_{2}$  & $e_{1}$ & $e_{2}$  
& $e_{1}$ & $e_{2}$  & $e_{1}$ & $e_{2}$  \\
\hline
\multicolumn{10}{|c|}{Test 2a}\\
\hline
1 & glo & 8.39 & 30.05 & 0.01 & 0.13 & 2.58 & 16.08 & 0.01 & 0.24 \\
 & loc1 & 97.02 & 98.41 & 93.47 & 95.40 & 97.28 & 98.35 & 93.26 & 95.35 \\
 & loc2 & 88.54 & 94.39 & 73.43 & 84.05 & 86.85 & 92.86 & 67.87 & 80.72 \\
2 & glo & 4.62 & 21.89 & 0.01 & 0.16 & 1.36 & 11.35 & 0.01 & 0.51 \\
 & loc1 & 95.52 & 97.48 & 91.55 & 94.39 & 93.50 & 96.12 & 88.83 & 92.87 \\
 & loc2 & 76.27 & 87.89 & 61.45 & 77.20 & 64.74 & 80.07 & 39.28 & 61.48 \\
4 & glo & 2.75 & 16.20 & 0.01 & 0.31 & 0.70 & 7.93 & 0.00 & 0.71 \\
 & loc1 & 90.57 & 94.83 & 90.19 & 93.73 & 81.60 & 89.31 & 83.39 & 89.97 \\
 & loc2 & 55.06 & 74.79 & 50.93 & 70.01 & 17.93 & 42.51 & 19.86 & 43.94 \\
8 & glo & 1.38 & 11.30 & 0.01 & 0.52 & 0.34 & 5.46 & 0.01 & 0.84 \\
 & loc1 & 73.98 & 85.46 & 88.77 & 93.08 & 47.22 & 67.75 & 75.29 & 85.37 \\
 & loc2 & 18.17 & 43.20 & 36.66 & 59.82 & 1.18 & 11.31 & 10.52 & 31.80 \\
16 & glo & 0.73 & 7.98 & 0.00 & 0.70 & 0.16 & 3.71 & 0.01 & 0.94 \\
 & loc1 & 33.13 & 57.27 & 73.88 & 84.25 & 5.24 & 22.73 & 52.41 & 71.38 \\
 & loc2 & 2.51 & 16.51 & 16.15 & 39.94 & 0.17 & 3.87 & 1.16 & 10.73 \\
32 & glo & 0.36 & 5.52 & 0.01 & 0.85 & 0.08 & 2.51 & 0.01 & 1.03 \\
 & loc1 & 3.15 & 17.96 & 50.07 & 69.75 & 0.09 & 2.98 & 20.16 & 44.28 \\
 & loc2 & 0.42 & 6.36 & 7.02 & 26.33 & 0.08 & 2.51 & 0.06 & 2.41 \\
\hline
\multicolumn{10}{|c|}{Test 2b}\\
\hline
1 & glo & 0.87 & 8.46 & 0.00 & 0.12 & 0.30 & 4.86 & 0.01 & 0.82 \\
 & loc1 & 99.36 & 99.77 & 99.87 & 99.82 & 99.83 & 99.92 & 99.88 & 99.85 \\
 & loc2 & 98.14 & 99.33 & 99.66 & 99.62 & 99.31 & 99.71 & 99.50 & 99.61 \\
2 & glo & 0.38 & 5.53 & 0.00 & 0.56 & 0.24 & 4.87 & 0.15 & 3.79 \\
 & loc1 & 99.88 & 99.92 & 99.84 & 99.79 & 99.79 & 99.88 & 99.81 & 99.80 \\
 & loc2 & 97.68 & 99.11 & 99.37 & 99.45 & 99.06 & 99.56 & 98.47 & 99.15 \\
4 & glo & 0.29 & 4.76 & 0.03 & 1.68 & 0.27 & 5.31 & 0.47 & 6.70 \\
 & loc1 & 99.52 & 99.77 & 99.74 & 99.75 & 99.62 & 99.78 & 99.73 & 99.73 \\
 & loc2 & 96.40 & 98.55 & 97.98 & 98.86 & 97.07 & 98.68 & 96.52 & 98.19 \\
8 & glo & 0.23 & 4.58 & 0.18 & 4.21 & 0.31 & 5.58 & 0.67 & 8.05 \\
 & loc1 & 99.32 & 99.63 & 99.76 & 99.75 & 98.51 & 99.27 & 99.64 & 99.67 \\
 & loc2 & 93.41 & 97.15 & 96.97 & 98.43 & 83.22 & 92.36 & 93.68 & 96.82 \\
16 & glo & 0.26 & 5.18 & 0.40 & 6.22 & 0.33 & 5.73 & 0.86 & 9.15 \\
 & loc1 & 97.66 & 98.81 & 99.74 & 99.74 & 93.61 & 96.63 & 98.94 & 99.36 \\
 & loc2 & 84.35 & 92.52 & 95.21 & 97.75 & 5.41 & 23.44 & 75.74 & 87.93 \\
32 & glo & 0.29 & 5.42 & 0.65 & 7.98 & 0.33 & 5.71 & 1.10 & 10.34 \\
 & loc1 & 90.70 & 95.13 & 98.97 & 99.40 & 4.99 & 22.53 & 95.17 & 97.62 \\
 & loc2 & 44.02 & 66.49 & 86.38 & 93.70 & 0.33 & 5.72 & 7.54 & 28.00 \\
\hline
\end{tabular}
\caption{Tests 2a and 2b. MsGR with global and localized multiscale basis functions ($\delta_H=0.1$ for loc1 and $0.2$ for loc2) }
\label{table:off-t2}
\end{table}

 Therefore, we solve the following formulation: 
\[
- \nabla \cdot (k_{\perp} \nabla u)
- \nabla \cdot ( k_{\Delta} \ {b} ({b} \cdot \nabla u) )
 = f, 
\]
with $k_{\Delta} = k_{\parallel} - k_{\perp}$. This is a simplified model for temperature evolution along magnetic fields in magnetic confinement fusion, which remains a very challenging problem due to dynamically evolving non-mesh-aligned anisotropy and anisotropy ratios of $10^{10}$ or larger, e.g., \cite{Wimmer2025}.

Homogeneous Dirichlet boundary conditions $u(x) = 0$ are imposed on $\partial\Omega$. We set $f=0$ and consider $k_{\parallel} = 1.0$ and $k_{\perp} = 10^{-1}$ in \textit{Test 2a} and $k_{\perp} = 10^{-3}$ in \textit{Test 2b}.  We use a structured $200 \times 200$ mesh with 80,000 cells, 120,400 edges, and 40,401 nodes.
Figure \ref{fig:test2-msbasis} shows the multiscale basis functions obtained using the CF and MC approaches with both global and localized constructions. We observe that, in both approaches, the basis functions align with the magnetic field. 
In Figures \ref{fig:test2},  we present a multiscale solution in coarse- and fine-scale representations. The first row illustrates the solutions at the coarse degrees of freedom as well as their magnetic field-aligned distribution of centroids (coarse nodes). The corresponding coarse-scale solutions are then projected onto the fine-scale resolution using the prolongation matrix $P$, and the results are plotted in the second row. 

In Table \ref{table:off-t2}, we present relative errors in \% for a varying number of basis functions $M$. We present results for global (glo) and localized basis construction, where we consider local domains corresponding to oversampled local domains $\Omega^+_i$ constructed with $\delta_H=0.1$ (loc1) and $\delta_H=0.2$ (loc2). We observe that a larger local domain is needed to achieve good localization due to the strong anisotropy. 
We see the error decay for different coarse space constructions as the number of multiscale basis functions (local clusters) per local domain increases. The plots compare the CF-approach and the MC-approach with global and localized basis functions for both moderate (Test 2a) and strong (Test 2b) anisotropies. For moderate anisotropy, both methods demonstrate rapid error reduction as $M$ increases. However, as anisotropy becomes larger in Test 2b, the effectiveness of localized methods diminishes for small $M$, and accurate coarse spaces only emerge when $M$ is sufficiently large and more global information is incorporated by considering a larger oversampled local domain in basis construction. The challenge of resolving highly anisotropic features with localized bases emphasizes the need for a global approach in the extreme anisotropy regime.

\subsection{Test 3: Pore network model}

We consider a time-dependent problem on a discrete structure defined by a pore network model, where assign a heterogeneous property $c_i$ for each node that can be associated with the pore volume or fluid compressibility \cite{bluntpnm}. 
The mass conservation equation leads to the following parabolic equation
\[
C u_t  + A u = f,   
\quad 0 < t \leq T,
\]
where $C = \text{diag}(c_1, \ldots , c_n)$ is the diagonal matrix. We set point-source at the two (Test 3a) and one (Test 3b) boundary points connected to channels and consider initial conditions $u = u_0$ at $t = 0$. 

\begin{figure}[h!]
\centering
\begin{minipage}[t]{0.32\textwidth} 
\centering
\includegraphics[width=1\textwidth]{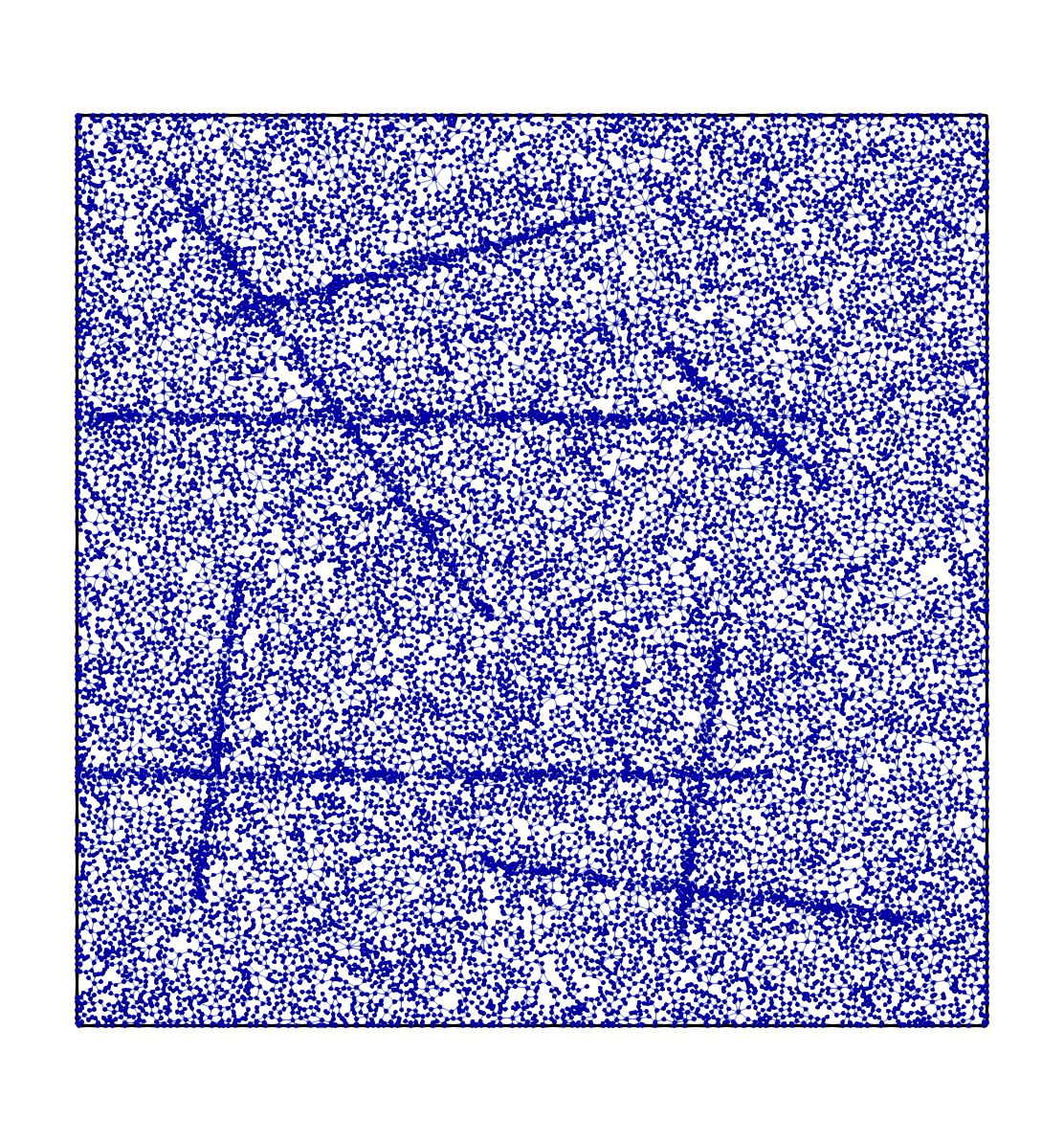}
\subcaption{Test 3a}
\end{minipage}
\ \ \ \ \ \ \ \ \ \ \
\begin{minipage}[t]{0.42\textwidth} 
\centering
\includegraphics[width=1\textwidth]{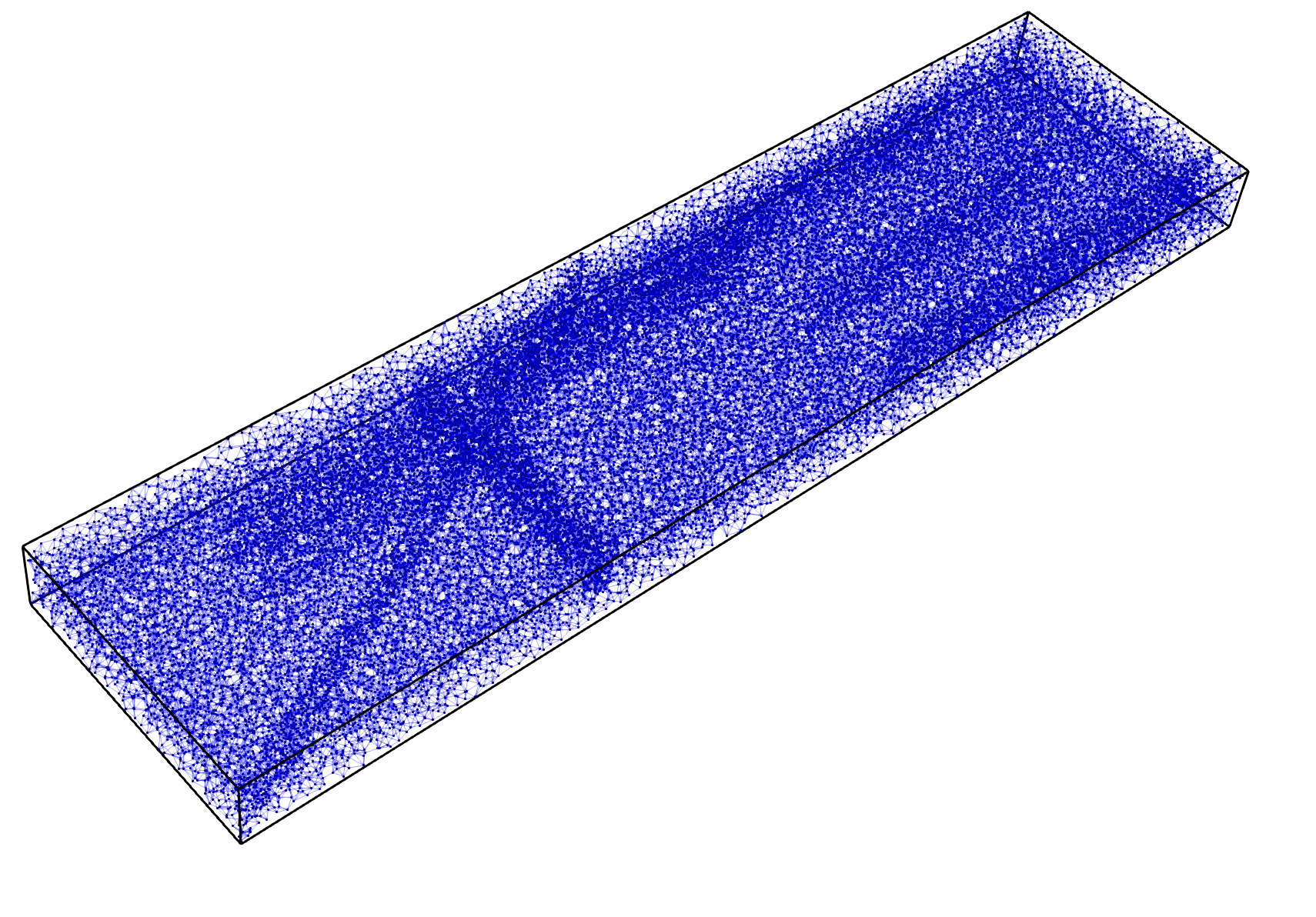}
\subcaption{Test 3b}
\end{minipage}
\caption{{Pore network structures used in Tests 3a and 3b. 
Left: Test 3a is a two-dimensional network embedded in square domain with 18,983 nodes and 27,936 connections. 
Right: Test 3b is a three-dimensional network embedded in rectangular domain with 38,058 nodes and 113,000 connections.}}
\label{fig:test3g}
\end{figure}

\begin{figure}[h!]
\centering
\begin{minipage}[t]{0.43\textwidth} 
\centering
\includegraphics[width=1\textwidth]{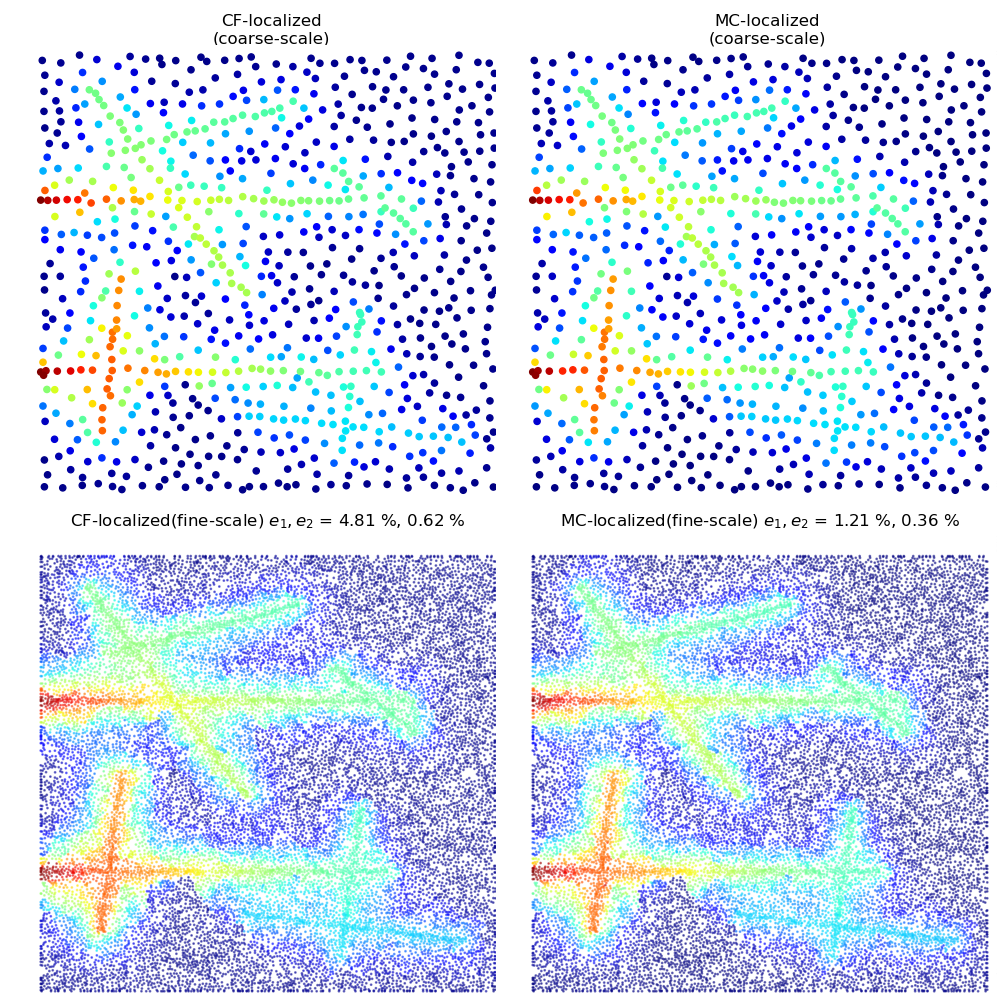}
\subcaption{Test 3a}
\end{minipage}
\ \ \ \ 
\begin{minipage}[t]{0.53\textwidth} 
\centering
\includegraphics[width=1\textwidth]{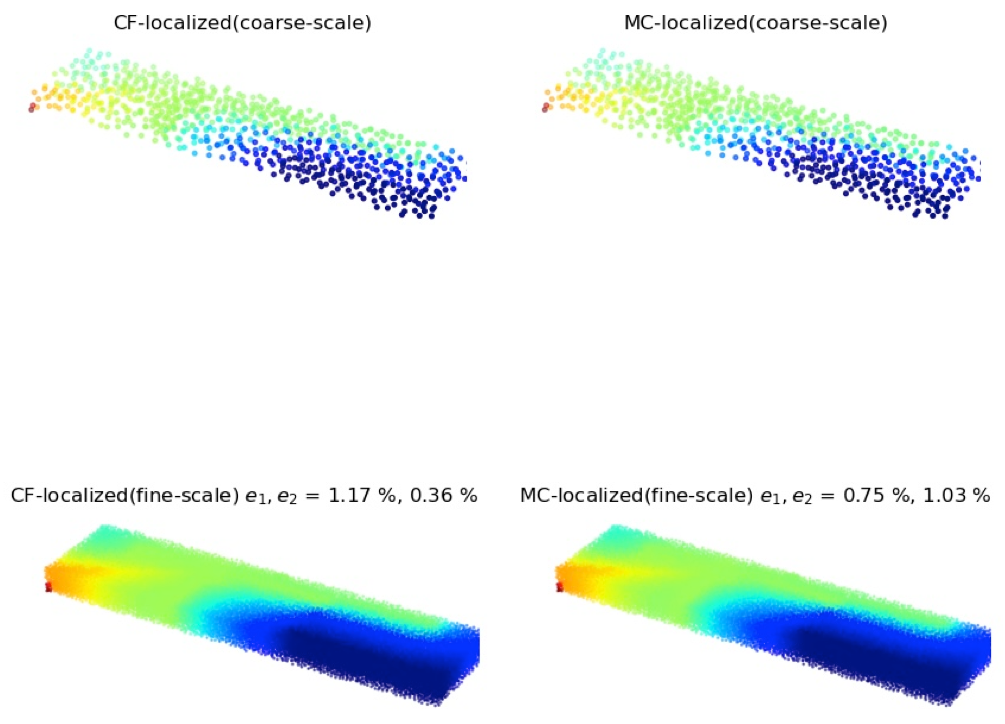}
\subcaption{Test 3b}
\end{minipage}
\caption{Tests 3a and 3b. Multiscale solution: coarse-scale solution $u_c$ (first row) and projected to fine-scale $u_{ms}$ (second row).
Test 3a: $DOF_c = 802$ (coarse-scale, $N_{\Omega}=100$ and $M=8$) and $DOF = 18,983$ (fine-scale). 
Test 3b: $DOF_c = 801$ (coarse-scale, $N_{\Omega}=100$ and $M=8$) and $DOF = 38,058$ (fine-scale)}
\label{fig:test3}
\end{figure}

\begin{table}[ht!]
\centering
\begin{tabular}{|cc||cc|cc||cc|cc||}
\hline
\multirow{4}{*}{M} & 
& \multicolumn{4}{c||}{$N_{\Omega} = 25$} 
& \multicolumn{4}{c||}{$N_{\Omega} = 100$}\\ 
 & 
& \multicolumn{2}{c}{CF} & \multicolumn{2}{c||}{MC} 
& \multicolumn{2}{c}{CF} & \multicolumn{2}{c||}{MC}  \\
& 
& $e_{1}$ & $e_{2}$  & $e_{1}$ & $e_{2}$ 
& $e_{1}$ & $e_{2}$  & $e_{1}$ & $e_{2}$ \\
\hline
\multicolumn{10}{|c|}{Test 3a}\\
\hline
1 & glo & 64.95 & 6.71 & 18.35 & 1.89 & 23.21 & 2.54 & 4.62 & 0.49 \\
 & loc & 66.09 & 7.35 & 54.78 & 6.33 & 53.26 & 6.29 & 74.39 & 8.72 \\
2 & glo & 45.78 & 4.80 & 11.10 & 1.15 & 15.74 & 1.77 & 2.43 & 0.26 \\
 & loc & 54.34 & 6.00 & 49.75 & 5.79 & 30.33 & 3.63 & 59.90 & 7.05 \\
4 & glo & 26.19 & 2.83 & 6.60 & 0.69 & 9.60 & 1.12 & 0.89 & 0.10 \\
 & loc & 26.76 & 3.03 & 40.55 & 4.83 & 9.60 & 1.18 & 32.68 & 4.00 \\
8 & glo & 17.89 & 1.97 & 3.41 & 0.36 & 4.81 & 0.62 & 0.29 & 0.04 \\
 & loc & 17.80 & 1.96 & 8.67 & 1.38 & 4.81 & 0.62 & 1.21 & 0.36 \\
16 & glo & 9.46 & 1.10 & 1.03 & 0.12 & 2.10 & 0.32 & 0.09 & 0.02 \\
 & loc & 9.46 & 1.10 & 1.14 & 0.14 & 2.10 & 0.32 & 0.12 & 0.04 \\
32 & glo & 4.82 & 0.62 & 0.32 & 0.04 & 0.96 & 0.18 & 0.03 & 0.01 \\
 & loc & 4.82 & 0.62 & 0.33 & 0.05 & 0.96 & 0.18 & 0.03 & 0.01 \\
\hline
\multicolumn{10}{|c|}{Test 3b}\\
\hline
1 & glo & 11.15 & 2.78 & 2.20 & 0.84 & 5.82 & 1.17 & 0.63 & 0.35 \\
 & loc & 82.22 & 30.74 & 83.06 & 32.08 & 82.26 & 30.73 & 86.57 & 33.15 \\
2 & glo & 7.68 & 1.51 & 0.80 & 0.43 & 4.11 & 0.87 & 0.33 & 0.23 \\
 & loc & 81.85 & 30.74 & 79.01 & 30.51 & 15.82 & 7.15 & 71.34 & 27.51 \\
4 & glo & 6.29 & 1.25 & 0.29 & 0.22 & 2.32 & 0.56 & 0.10 & 0.13 \\
 & loc & 17.47 & 8.23 & 52.28 & 20.49 & 1.64 & 0.59 & 40.90 & 16.53 \\
8 & glo & 4.09 & 0.87 & 0.22 & 0.18 & 1.24 & 0.38 & 0.03 & 0.10 \\
 & loc & 3.46 & 1.49 & 50.01 & 19.67 & 1.17 & 0.36 & 0.75 & 1.03 \\
16 & glo & 2.39 & 0.57 & 0.09 & 0.12 & 0.56 & 0.29 & 0.02 & 0.14 \\
 & loc & 2.32 & 0.56 & 12.66 & 6.18 & 0.56 & 0.29 & 0.08 & 0.28 \\
32 & glo & 1.25 & 0.38 & 0.02 & 0.09 & 0.23 & 0.32 & 0.02 & 0.18 \\
 & loc & 1.24 & 0.38 & 0.54 & 0.84 & 0.23 & 0.32 & 0.03 & 0.19 \\
\hline
\end{tabular}
\caption{Tests 3a and 3b. MsGR with global and localized multiscale basis functions}
\label{table:off-t3}
\end{table}

To define a coarse grid approximation of unsteady system, we form a coarse-grained system as follows
\[
C_c (u_c)_t + A_c u_c = f_c \quad  \text{with} \quad 
A_c = R A P, \quad C_c = R C P, \quad f_c = R f\quad  \text{and} \quad u_{ms} = P u_c.
\]
We let $\tau$ be an uniform time step size and $u^n$ be a solution at time $t^n$, where $t^n = \ell \tau$ for $\ell =1,\ldots,N_t$. Then, using the implicit backward Euler approximation for the time derivative, we obtain the following fully discrete coarse-grained system
\[
C_c \frac{u_c^{n+1} - u_c^n}{\tau} + A_c^{n+1} u_c^{n+1} = f_c, \quad 
A_c = R A P, \quad C_c = R C P, \quad f_c = R f, \quad
u_{ms} = P u_c.
\]
with initial condition, $u_c^{n=0} = u_{c,0}$. We simulate for $T = 10^2$ with 20 time steps. 

We consider two test cases corresponding to two pore-structures (Figure \ref{fig:test3g}). In Test 3a, the geometry and connectivity of the network are extracted from two images with differing porosity using a network generation algorithm from the {OpenPNM} library \cite{gostick2016openpnm}. The resulting pore network in Test 3a consists of 18,983 nodes and 27,936 connections, corresponding to a structure embedded in a square domain, {$[0, 3000]^2$}. 
In Test 3b, we have a network with 38,058 nodes and 113,000 connections, corresponding to a structure embedded in a square domain, {$[0, 1] \times [0, 0.25] \times [0, 0.065]$}. 
Connection weights are assigned based on Poiseuille flow through tubes, using the geometric characteristics of the network, with values ranging from $0.0008 \leq w_{ij} \leq 18.85$ in Test 3a and from $0.0016 \leq w_{ij} \leq 46.64$ in Test 3b. For coefficients $c_i$, the values range from $0.1 \leq c_{i} \leq 0.82$ in Test 3a and from $0.1 \leq c_{i} \leq 1$ in Test 3b. 

In Figures \ref{fig:test3},  we present a multiscale solution in coarse- and fine-scale representations. 
In Table \ref{table:off-t3}, we present relative errors in \% for varying number of basis functions $M$ for  $N_{\Omega} = 25$ and 100. We present relative errors in $L_2$ ($e_1$) and energy norm ($e_2$) at the final time. 
We observe an error decrease for larger M and $N_{\Omega}$ as well as a good localization with a larger number of local clusters.

\section{Conclusion}

We have proposed a multiscale graph-based reduction method for upscaling heterogeneous and anisotropic diffusion problems. The method combines balanced domain decomposition with local spectral clustering to construct multiple coarse-scale variables that effectively capture key local features. Two strategies for the construction of a prolongation operator and the assembly of a coarse-scale system have been developed. 
Numerical experiments on diverse test cases are presented, including perforated domains, channelized media, strongly anisotropic coefficients, and discrete pore-network models, which confirm the robustness and efficiency of the proposed upscaling framework. 
We find that the global MC approach yields the most accurate coarse model, even with a single vector and the original subdomains without clustering. In this setting, clustering plays a crucial role in localization and can be used to sparsify the coarse basis when combined with energy minimization. Although the global MC approach requires solving a saddle-point system to construct the energy-minimizing basis, it remains practical with aggressive coarsening and approximate solvers. This suggests that the method warrants further study, particularly with more efficient approximations~\cite{scheichl}. For the local approaches, the MC method yields the best results; however, effective localization requires a larger number of local basis functions. By contrast, the CF approach benefits more directly from localization, while the MC approach becomes competitive only for larger numbers of basis functions (clusters and centroids).
{The two approaches serve different purposes: MC provides a more accurate averaging-based upscaled model in many tests, while CF provides a simpler C-point-based construction that is directly connected to classical AMG interpolation and is attractive for localized and solver-oriented settings.}
{In future work, we will investigate alternative coarse-space constructions and multilevel extensions of the proposed approach.}

\section*{Acknowledgments}
BSS was supported by the DOE Office of Advanced Scientific Computing Research Applied Mathematics program through Contract No. 89233218CNA000001. LANL report number LA-UR-25-29298.

\bibliographystyle{unsrt}
\bibliography{lit}

\clearpage
\appendix
\section{Convergence analysis}\label{app1}
In this section, we derive convergence estimates for both proposed coarsening approaches.
We introduce the discrete weighted norms  \cite{vasilyeva2025generalized, vasilyeva2025msaniso} 
\[
\|v\|^2_{D} = v^T D v 
= \sum_{i \in \mathcal{V}} d_i v_i^2, \quad 
\|v\|^2_{L} = v^T L v 
 = \sum_{\substack{(i,j)\in \mathcal{E} \\ i,j \in \mathcal{V}}} 
w_{ij} (v_i-v_j)^2,\quad 
\|v\|^2_{A} = v^T A v.
\]

For \eqref{eq:coarse}, we have $u_{ms} = P u_c = P(RAP)^{-1}Rf$, then
\[
u - u_{ms} = u - P(RAP)^{-1}Rf = u - P(RAP)^{-1}RAu
= (I - P(RAP)^{-1}RA)u = (I - \Pi) u, 
\]
with $\Pi = P(RAP)^{-1}RA$.  
We have the Galerkin orthogonality
\[
R A (u - u_{ms}) = R (Au - Au_{ms}) = R(f - Au_{ms}) = 0,
\]
since $A_c u_c - f_c = RAPu_c - Rf = R(A u_{mc} - f) = 0$ and 
\[
\Pi (u - u_{ms}) =  P(RAP)^{-1}RA(u - u_{ms}) = 0.
\]

Let $I_k = I^{\Omega_k^{+,h}} \in \mathbb{R}^{n_k\times n}$denote the restriction operator onto $\Omega_k^+$,
and $A^{\Omega_k^+} = I_k A I_k^T \in \mathbb{R}^{n_k\times {n_k}}$, $n_k = |\Omega_k^+|$. 
With $P = [ I_1^T P_1,\dots,I_{N_{\Omega}}^T P_{N_{\Omega}} ]$, the Galerkin orthogonality is equivalent to the block relations
\[
R_k I_k A (u-u_{ms}) = R_k \big( I_k A I_k^T \big) I_k (u-u_{ms}) 
=
R_k A^{\Omega_k^+}\, I_k (u-u_{ms}) = 0, 
\quad k=1,\dots,N_{\Omega}.
\]
Then 
\[ 
\pi_k \big(I_k (u-u_{ms})\big) = 0
\quad \text{ with } \quad 
\pi_k = P_k\big(R_k A^{\Omega_k^+} P_k\big)^{-1} R_k A^{\Omega_k^+}.
\]

The proposed multiscale method begins by defining coarse variables, which are constructed using local spectral clustering within the non-overlapping subdomains $\Omega_k$ (subdomains). Local spectral clustering divide $\Omega_k$ into $r$ non-oversampling subsubdomains $\mathcal{A}_k^{r}$ ($\Omega_k = \cup_{r=1}^{M_k} \mathcal{A}_k^{r}$) and ensures that the contrast of the edge weights $w_{ij}$, and consequently the degrees $d_i$, within each local cluster $\mathcal{A}_k^{r}$ is small. 

The bounds for a local cluster $\mathcal{A}_k^{r}$ define local ratio
\[
C_{{ratio}}^{k,r} = \frac{\overline{d}^{r,k}}{\underline{d}^{r,k}} \approx \frac{\overline{w}^{r,k}}{\underline{w}^{r,k}}, \quad 
\underline{w}^{k,r} \leq w_{ij} \leq \overline{w}^{k,r}, \quad 
\underline{d}^{k,r} \leq d_i \leq \overline{d}^{k,r}, \quad i,j \in \mathcal{A}_k^r,
\]
where $C_{{ratio}}^{k,r}$ captures the intra-cluster weight/degree contrast and $d_i= -\sum_j w_{ij}$.
For special cases, such as channelized media, we may have $C_{{ratio}}^{k,r} \approx 1$, indicating that the weights within each cluster are nearly uniform and exhibit similar connectivity patterns. In anisotropic media, the clustering aligns with the anisotropy directions.
The choice of $M_k$ (number of local clusters) directly influences the quality of the clustering and the resulting bounds. As $M_k$ increases, each cluster $\mathcal{A}_k^{r}$ becomes smaller and more homogeneous, leading to smaller values of $C_{ratio}^{k,r}$. 

We let $H_{k,r} = \text{diam}(\mathcal{A}_{k}^r)$ and $C_p > 0$, then using the Poincar\'e inequality, we obtain
\begin{equation}
\|v\|_{D^{\mathcal{A}_{k}^{r}}}^2 \leq 
\overline{d}^{r,k} \|v\|_k^2 = 
\underline{d}^{r,k} \frac{\overline{d}^{r,k}}{\underline{d}^{r,k}} \|v\|_k^2  \leq 
\underline{d}^{r,k}  C_{{ratio}}^{k,r} C_p^2  H_{k,r}^2 \|\nabla v\|^2_k
\leq 
 C_{{ratio}}^{k,r} C_p^2  H_{k,r}^2 \|v\|^2_{L^{\mathcal{A}_{k}^{r}}},
 \label{eq:locpoincare}
\end{equation}
where 
$\|v\|_k^2 = \sum_{i \in \mathcal{A}_{k}^{r}} v_i^2$ and 
$\|\nabla v\|^2_k = \sum_{\substack{(i,j)\in \mathcal{E} \\ i,j \in \mathcal{A}_{k}^{r}}} (v_i-v_j)^2$ are the unweighted and discrete gradient norms \cite{coulhon1998random, badr2012weighted}. 
Then for  $\Omega_k=\bigcup_{r} \mathcal{A}_k^r$, we have \cite{brannick2019role, chung2018constraint, zhao2020analysis}
\begin{equation}
\|v\|_{D^{\Omega_k}}^2
\le C_{\mathrm{ratio}}^k C_p^2 H_k^2
\sum_r \|v\|_{L^{\mathcal A_k^r}}^2
= C_{\mathrm{ratio}}^k C_p^2 H_k^2 \|v\|_{L^{\Omega_k}}^2
\leq C_{\mathrm{ratio}}^k\,C_p^2\,H_k^2  \|v\|_{A^{\Omega_k}}^2,
\label{eq:lap}
\end{equation}
with 
$k_{\min}^k := \min_r k_{\min}^{\,k,r}$, 
$C_{\mathrm{ratio}}^k = \max_r C_{\mathrm{ratio}}^{k,r}$ and $H_k := \max_r H_{k,r}$.

 Additionally in CF-approach, we have the following block structure \cite{manteuffel2019convergence, brannick2019role}
\[
P =  \begin{bmatrix}
W \\
I
\end{bmatrix}, 
\quad 
R =  \begin{bmatrix}
Z & I
\end{bmatrix}, 
\quad 
A = 
\begin{bmatrix}
A_{ff} & A_{fc} \\
A_{cf} & A_{cc}
\end{bmatrix}
=
\begin{bmatrix}
I & 0 \\
A_{cf}A_{ff}^{-1} & I
\end{bmatrix}
\begin{bmatrix}
A_{ff} & 0 \\
0 & S
\end{bmatrix}
\begin{bmatrix}
I & A_{ff}^{-1}A_{fc} \\
0 & I
\end{bmatrix}, 
\]
with $W = -A_{ff}^{-1} A_{fc}$,  $Z =-A_{cf} A_{ff}^{-1}$ and $S = A_{cc} - A_{cf}A_{ff}^{-1}A_{fc}$. 
Then, we have $RAP = S$  and 
\[
\Pi = P(RAP)^{-1}RA =
\begin{bmatrix}
0 & W \\
0 & I
\end{bmatrix}, 
\quad  \text{then} \quad  
(I - \Pi) e = \begin{bmatrix}
I & -W \\
0 & 0
\end{bmatrix}
\begin{bmatrix}
e_f \\
e_c
\end{bmatrix}
 = 
 \begin{bmatrix}
 e_f - W e_c\\
 0
 \end{bmatrix},  
 \quad 
 e = u- u_{ms}.
\]

Therefore based on this special structure of the error in CF-approach, we have \cite{stuben1999algebraic}
\[
\| (I - \Pi) e\|^2_A 
= e^T A e - e_c^T S e_c \leq ||e||^2_A,
\quad S = A_{cc} - A_{cf}A_{ff}^{-1}A_{fc}.
\]

\begin{lemma}\label{lem1}
Let $u$ be solution of \eqref{eq:sys-f} and $u_{ms}$ be a solution of \eqref{eq:coarse}, then we have
\begin{equation}
\|u - u_{ms}\|_D^2  \leq 
C^2 H^2 C_{{ratio}} \|u - u_{ms}\|^2_A,
\label{eq:sap}
\end{equation}
with $H = \max_{k,r} H_{k,r}$ and $C_{{ratio}} = \max_{k,r} C_{{ratio}}^{k,r}$. 

\end{lemma}
\begin{proof}
Since the subdomains $\{\Omega_k\}$ are non-overlapping and based on the Galerkin orthogonality, we have
\[
\|u-u_{ms}\|_D^2  = \|e\|^2_D 
=
\sum_{k=1}^{N_{\Omega}} \|e_k\|_{D^{\Omega_k^+}}^2  
 = 
\sum_{k=1}^{N_{\Omega}} \|(I - \pi_k) e_k\|_{D^{\Omega_k^+}}^2  
 = 
\sum_{k=1}^{N_{\Omega}} \|r_k\|_{D^{\Omega_k^+}}^2 ,
\]
with
$e = u-u_{ms}$, $e_k = I_k e$ and  $r_k = (I - \pi_k) e_k$.

For the CF-approach, we have
\[
 \sum_{k=1}^{N_{\Omega}} \|r_k\|_{D^{\Omega_k^+}}^2  
=
\sum_{k=1}^{N_{\Omega}}  \sum_{j=1}^{\Omega_j \in \Omega_k^+}  \|r_k\|_{D^{\Omega_j}}^2
\leq C_o \sum_{k=1}^{N_{\Omega}}   \|r_k\|_{D^{\Omega_k}}^2
\]
where $C_o$ depends on  the maximum number of overlapping subdomains sharing the same $\Omega_k$.

With the MC approach and the basis construction constraints, we obtain  \cite{chung2018constraint, zhao2020analysis}
\[
 \sum_{k=1}^{N_{\Omega}} \|r_k\|_{D^{\Omega_k^+}}^2  
 =
\sum_{k=1}^{N_{\Omega}}   \|r_k\|_{D^{\Omega_k}}^2
\]
Then using the local estimate \eqref{eq:lap} on each $\Omega_k$, we obtain
\[
\|e\|_D^2 
\leq  C_o 
\sum_{i=1}^{N_{\Omega}}  
C_{\mathrm{ratio}}^k\,C_p^2\,H_k^2 \|r_k\|_{A^{\Omega_k}}^2
\leq  
C^2 H^2 C_{{ratio}} \|r\|^2_A
\leq  
C^2 H^2 C_{{ratio}} \|e\|^2_A,
\]
with $C^2=C_o C_p^2$ ($C_o=1$ in MC-approach), $H = \max_{k} H_{k}$ and $C_{{ratio}} = \max_{k} C_{{ratio}}^{k}$.
\end{proof}

\begin{theorem}\label{terr}
Let $u$ be solution of \eqref{eq:sys-f} and $u_{ms}$ be a solution of \eqref{eq:coarse}, then we have
\begin{equation}
\|u - u_{ms}\|_A \leq C H C_{{ratio}}^{1/2} \|f\|_{D^{-1}}.
\label{eq:est-t1}
\end{equation}
\end{theorem}
\begin{proof}
Based on \eqref{eq:sys-f} and \eqref{eq:coarse} with $A = A^T$, $R=P^T$, we have
\[
(u - u_{ms})^T A u_{ms} = (Au)^T P u_c -  u_c^T R A P u_c 
= (Au)^T P u_c -  u_c^T R f = (Au)^T P u_c -  (P u_c)^T f = 0.
\]
Therefore using the Cauchy–Schwarz inequality, we obtain
\[
\|u - u_{ms}\|^2_A = (u - u_{ms})^T A (u - u_{ms}) = 
(u - u_{ms})^T A u = (u - u_{ms})^T f \leq 
\|f\|_{D^{-1}} \ \|u - u_{ms}\|_D. 
\]
Finally using the Lemma \ref{lem1}, we get
\[
\|u - u_{ms}\|^2_A \leq 
\|f\|_{D^{-1}} 
\|u - u_{ms}\|_D  \leq 
\|f\|_{D^{-1}} 
C H C_{{ratio}}^{1/2} \|u - u_{ms}\|_A.
\]
\end{proof}

\section{{Two-level solver with local correction iterations and graph coloring}}\label{app2}

{
We consider a two-level iterative method composed of two key components: (i) a multiscale graph-based coarse solve and (ii) relaxation based on local correction iterations performed over colored aggregates. 
The localization of the correction steps is similar to our approach for computing subdomains, where the computational domain is partitioned into nonoverlapping subdomains $\Omega_k$ and the correction is computed locally within each oversampled subdomain $\Omega_i^{+,h}$ constructed by adding a few fine-scale layers. 
The overall iterations are computed as in a Schwarz-type method, where we solve the following local problems in each local domain $\Omega_k^{+,h}$ ($\forall k = 1,\ldots,N_{\Omega}$)
\[
A^{\Omega_k^{+,h}} w^{\Omega_k^{+,h}} = r^{\Omega_k^{+,h}}, \quad 
A^{\Omega_k^{+,h}} = I^{\Omega_k^{+,h}} A \Big(I^{\Omega_k^{+,h}} \Big)^T, \quad 
r^{\Omega_k^{+,h}} = I^{\Omega_k^{+,h}} r, 
\]
where  $I^+_k$ is the restriction operator that maps local degrees of freedom in the subdomain $\Omega_k^{+,h}$ to global, $A^{\Omega_k^{+,h}}$ is the restriction of the global system matrix $A$ to subdomain $\Omega_k^{+,h}$ and  $r^{\Omega_k^{+,h}}$ is the restriction of the global residual $r = f - L u$ to the local domain $\Omega_k^{+,h}$.

\begin{figure}[!h]
\centering
\includegraphics[width=0.45\textwidth]{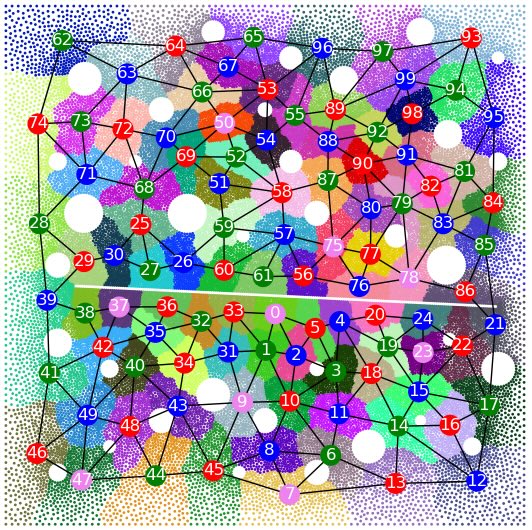}
\caption{{Subdomains $\Omega_k$ and coloring with $N_{\Omega} = 100$. Each subdomain corresponds to a node in the partition graph $G_C = (\mathcal{V}_C,\mathcal{E}_C)$. The nodes $\mathcal{V}_C$ are colored into four groups (red, blue, green, and pink).}}
\label{fig:color1}
\end{figure}

We use coloring to organize fractional iterations that parallelize the multiplicative Schwarz method. 
To define groups (colors), we construct a partition graph $G_C = (\mathcal{V}_C,\mathcal{E}_C)$, where each node in $\mathcal{V}_C$ corresponds to an subdomain $\Omega_k$, and an  edge $(\Omega_i, \Omega_j) \in \mathcal{E}_C$ is added if the subdomains $\Omega_i$ and $\Omega_j$ are connected, i.e., share a common interface. The graph $G_C$ encodes the interaction between subdomains and is used to group non-interacting subdomains (or weakly interacting) and assign a color to each group. Then we can iteratively update the solution. We use a greedy coloring strategy which  assigns a color to each subdomain such that no two adjacent subdomains share the same color, ensuring that corrections within each color class can be applied in parallel without conflict. 
In Figure \ref{fig:color1}, we show the partition graph $G_C$ with $N_{\Omega} = 100$ subdomains, where each node corresponds to an subdomain $\Omega_k$. The graph is colored using the NetworkX library \cite{hagberg2008exploring} with four colors, $N_{clr} = 4$. The red group represents the first color class containing 26 subdomains, blue corresponds to the second group with 28 subdomains, green indicates the third group with 29 subdomains, and violet denotes the fourth group comprising 17 subdomains.

Let $\mathcal{C}_{\ell}$ denote the set of subdomains with color $\ell$, and $N_{clr}$ denote the total number of colors used. Then, the local correction iteration proceeds color by color, applying local solves for all subdomains within each $\mathcal{C}_\ell$. 
The sequential iterations over groups can be represented as a fractional stepping 
\[
u^{(m+\ell/N_{clr})} 
= u^{(m+(\ell-1)/N_{clr})} +  S_{\ell}^{-1} r^{(m+(\ell-1)/N_{clr})},
\quad 
S_{\ell}^{-1} = \sum_{\Omega_k \in \mathcal{C}_{\ell}} \Big(I^{\Omega_k^{+,h}}\Big)^T W^{\Omega_k^{+,h}} \Big(A^{\Omega_k^{+,h}} \Big)^{-1} I^{\Omega_k^{+,h}}.
\]
Finally, the two-level solver combines the MsGR coarse space with the colored block Schwarz method.
At each iteration, a coarse-grid correction is first computed using the MsGR coarse space. The resulting approximation is then improved by ($N_{sm}$) local correction sweeps based on the colored block Schwarz method. For each color group, independent local problems are solved on the oversampled subdomains, and the corresponding corrections are accumulated and applied to update the solution. The process is repeated until convergence.
Next, based on the results for the coarse-scale MsGR solver, we consider three representative tests: a heterogeneous perforated domain in Test 1a, highly anisotropic diffusion in a square domain in Test 2b, and a three-dimensional pore-network model in Test 3b. 
}

\begin{table}[h!]
\centering
\begin{tabular}{|cc|cc|cc|}
\hline
\multicolumn{2}{|c|}{Test 1a} & \multicolumn{2}{c|}{Test 2b} & \multicolumn{2}{c|}{Test 3b} \\
\hline
$N_{it}$ & tm & $N_{it}$ & tm & $N_{it}$ & tm \\
\hline
4529 & 16.544 
& nc & - 
& 90 & 12.424 \\
\hline
\end{tabular}
\caption{{SA-AMG solver. nc indicates non-convergence, i.e., when ${N}_{it} > 5000$. tm indicate wall-clock time in seconds}}
\label{table:test-amg}
\end{table}

\begin{table}[h!]
\centering
\begin{tabular}{|c|c|cc|cc|cc|cc|cc|cc|}
\hline
\multirow{3}{*}{$M$} & \multirow{3}{*}{$N_{sm}$} & \multicolumn{4}{c|}{Test 1a} & \multicolumn{4}{c|}{Test 2b} & \multicolumn{4}{c|}{Test 3b} \\
\cline{3-14}
 & & \multicolumn{2}{c|}{CF} & \multicolumn{2}{c|}{MC} & \multicolumn{2}{c|}{CF} & \multicolumn{2}{c|}{MC} & \multicolumn{2}{c|}{CF} & \multicolumn{2}{c|}{MC} \\
\cline{3-14}
 & & $N_{it}$ & tm & $N_{it}$ & tm & $N_{it}$ & tm & $N_{it}$ & tm & $N_{it}$ & tm & $N_{it}$ & tm \\
\hline
 1 & 1 & nc & - & nc & - & nc & - & nc & - & nc & - & nc & - \\
  & 2 & nc & - & nc & - & nc & - & nc & - & nc & - & nc & - \\
  & 3 & nc & - & nc & - & nc & - & nc & - & nc & - & nc & - \\
\hline
 2 & 1 & nc & - & nc & - & nc & - & nc & - & nc & - & nc & - \\
  & 2 & nc & - & nc & - & nc & - & nc & - & 488 & 100.36 & nc & - \\
  & 3 & nc & - & nc & - & nc & - & nc & - & 403 & 87.13 & 499 & 131.66 \\
\hline
 4 & 1 & nc & - & nc & - & nc & - & nc & - & 384 & 111.49 & nc & - \\
  & 2 & nc & - & nc & - & nc & - & nc & - & 202 & 62.43 & 455 & 181.09 \\
  & 3 & nc & - & nc & - & nc & - & nc & - & 140 & 47.75 & 310 & 132.26 \\
\hline
 8 & 1 & 397 & 4.64 & nc & - & nc & - & nc & - & 111 & 47.62 & 239 & 131.70 \\
  & 2 & 201 & 2.43 & nc & - & nc & - & nc & - & 59 & 22.93 & 124 & 64.29 \\
  & 3 & 135 & 1.69 & 448 & 7.37 & nc & - & nc & - & 41 & 16.94 & 86 & 45.48 \\
\hline
 16 & 1 & 175 & 2.30 & 342 & 6.10 & nc & - & nc & - & 44 & 22.77 & 98 & 62.10 \\
  & 2 & 90 & 1.22 & 179 & 3.41 & nc & - & nc & - & 23 & 11.61 & 54 & 35.32 \\
  & 3 & 61 & 0.87 & 123 & 2.27 & nc & - & nc & - & 16 & 8.21 & 39 & 25.53 \\
\hline
 32 & 1 & 102 & 1.56 & 129 & 2.65 & nc & - & nc & - & 26 & 16.29 & 45 & 32.14 \\
  & 2 & 53 & 0.86 & 67 & 1.36 & nc & - & nc & - & 14 & 8.27 & 26 & 19.00 \\
  & 3 & 37 & 0.56 & 46 & 0.91 & 469 & 5.77 & nc & - & 10 & 6.21 & 19 & 13.67 \\
\hline
\end{tabular}
\caption{{Two-level solver. Number of iterations for CF- and MC-approaches with pointwise Jacobi smoother. 
$M$ is the number of local clusters and $N_{sm}$ is the number of smoothing iterations. nc indicates non-convergence, i.e., when ${N}_{it} > 500$. tm indicate wall-clock time in seconds}}
\label{tab:recombined-j}
\end{table}

\begin{table}[h!]
\centering
\begin{tabular}{|c|c|cc|cc|cc|cc|cc|cc|}
\hline
\multirow{3}{*}{$M$} & \multirow{3}{*}{$N_{sm}$} & \multicolumn{4}{c|}{Test 1a} & \multicolumn{4}{c|}{Test 2b} & \multicolumn{4}{c|}{Test 3b} \\
\cline{3-14}
 & & \multicolumn{2}{c|}{CF} & \multicolumn{2}{c|}{MC} & \multicolumn{2}{c|}{CF} & \multicolumn{2}{c|}{MC} & \multicolumn{2}{c|}{CF} & \multicolumn{2}{c|}{MC} \\
\cline{3-14}
 & & $N_{it}$ & tm & $N_{it}$ & tm & $N_{it}$ & tm & $N_{it}$ & tm & $N_{it}$ & tm & $N_{it}$ & tm \\
\hline
 1 & 1 & nc & - & nc & - & nc & - & nc & - & 46 & 39.29 & 82 & 63.76 \\
  & 2 & nc & - & nc & - & nc & - & nc & - & 28 & 33.94 & 62 & 82.44 \\
  & 3 & nc & - & nc & - & nc & - & nc & - & 21 & 39.16 & 50 & 98.14 \\
\hline
 2 & 1 & 188 & 3.02 & nc & - & nc & - & nc & - & 7 & 5.34 & 42 & 34.75 \\
  & 2 & 151 & 5.12 & nc & - & nc & - & nc & - & 5 & 6.77 & 35 & 50.81 \\
  & 3 & 138 & 5.62 & nc & - & nc & - & nc & - & 5 & 10.18 & 31 & 66.12 \\
\hline
 4 & 1 & 7 & 0.11 & nc & - & nc & - & nc & - & 4 & 3.50 & 14 & 13.74 \\
  & 2 & 6 & 0.15 & nc & - & nc & - & nc & - & 4 & 4.53 & 12 & 18.19 \\
  & 3 & 5 & 0.16 & nc & - & nc & - & nc & - & 4 & 6.51 & 10 & 22.59 \\
\hline
 8 & 1 & 6 & 0.10 & 7 & 0.15 & 383 & 9.04 & 477 & 15.24 & 5 & 3.65 & 5 & 5.11 \\
  & 2 & 5 & 0.13 & 7 & 0.19 & 302 & 9.89 & 436 & 18.57 & 4 & 4.68 & 4 & 5.77 \\
  & 3 & 5 & 0.16 & 7 & 0.27 & 273 & 12.68 & 424 & 22.41 & 3 & 6.17 & 4 & 6.04 \\
\hline
 16 & 1 & 5 & 0.09 & 4 & 0.09 & 15 & 0.37 & 93 & 3.35 & 4 & 3.54 & 4 & 5.36 \\
  & 2 & 4 & 0.10 & 4 & 0.11 & 11 & 0.32 & 85 & 3.96 & 3 & 3.92 & 3 & 5.60 \\
  & 3 & 4 & 0.19 & 4 & 0.15 & 9 & 0.40 & 83 & 5.00 & 3 & 4.75 & 3 & 5.17 \\
\hline
 32 & 1 & 5 & 0.09 & 4 & 0.10 & 8 & 0.14 & 10 & 0.34 & 4 & 3.68 & 4 & 3.91 \\
  & 2 & 4 & 0.10 & 4 & 0.11 & 6 & 0.13 & 9 & 0.38 & 3 & 3.53 & 3 & 3.76 \\
  & 3 & 4 & 0.13 & 4 & 0.16 & 5 & 0.16 & 8 & 0.43 & 3 & 4.72 & 3 & 4.89 \\
\hline
\end{tabular}
\caption{{Two-level solver. Number of iterations for CF- and MC-approaches with colored local correction iterations with 5 fine-scale oversampling layers (S5). 
$M$ is the number of local clusters and $N_{sm}$ is the number of smoothing iterations. nc indicates non-convergence, i.e., when ${N}_{it} > 500$. tm indicate wall-clock time in seconds}}
\label{tab:recombined-s5}
\end{table}

{
Table \ref{table:test-amg} presents the results obtained with Smoothed Aggregation Algebraic Multigrid (SA-AMG) solvers using PyAMG library \cite{bell2023pyamg} with default settings. 
We observe that, for the heterogeneous perforated domain in Test 1a, the SA-AMG solver requires a very large number of iterations, with $N_{it} > 4529$. For Test 2b with highly anisotropic diffusion, the SA-AMG solver does not converge within the prescribed limit of 5000 iterations. For the pore-network model in Test 3b, the SA-AMG solver converges, but still requires an average of ${N}_{it}=90$ iterations per time step.

Tables \ref{tab:recombined-j} and \ref{tab:recombined-s5} present the number of iterations and the corresponding solution times for iterative multiscale solver with pointwise Jacobi (J) and colored local correction iterations with 5 fine-scale oversampling layers (S5). For the coarse-scale solver, we consider ($N_{\mathcal{A}} = 100$) and employ localized multiscale basis functions with $\delta_H=0.05$ in Test 1a, $\delta_H=0.2$ (loc2) in Test 2b and $\delta_H=0.1$ in Test 3b. 
We note that we only do a post-smoothing iteration and the time of solution (tm) does not include the initialization step, where we compute basis functions and do a pre-inversion/pre-factorization of the local matrices. 
For Test 1a, we observe that the iteration count is low when the coarse-scale approximation is sufficiently accurate. Furthermore, colored local correction iterations substantially enhance convergence compared with pointwitse smoother, leading to faster overall solution times.
For Test 2b, the results for two-level solver demonstrate the robustness of the colored local correction over the classical SA-AMG solvers and pointwise Jacobi (J) method for both CF and MC approaches. For Test 2b, the pointwise Jacobi methods consistently did not converge in 500 iterations or converged with a number of iterations close to the maximum. In contrast, the iterative solver with local correction converges as the number of local clusters $M$ increases. The MC-approach with local correction generally requires more iterations but maintains reasonable computational efficiency. Test 2b is more challenging overall, with many configurations failing to converge. However, the performance improves as (M) increases, and convergence is achieved by the CF approach with local correction (S5).
For Test 3b, the SA-AMG solver performs well, and two-level solver achieves comparable convergence. However, colored local smoothing consistently reduces the iteration count relative to pointwise smoothing.  
Overall, the resulting two-level solver achieves convergence in only a few iterations when the coarse space is well constructed.
}

\end{document}